\documentclass[journal,twoside]{IEEEtran}
\usepackage{graphicx}
\usepackage{textcomp}
\usepackage[usenames,dvipsnames,svgnames,table]{xcolor}
\definecolor{steelblue}{RGB}{70,130,180}
\definecolor{warmred}{RGB}{244, 104, 65}
\usepackage[cmex10]{amsmath}
\usepackage{amssymb,amsfonts,mathrsfs,mathtools,amsthm}
\usepackage{epstopdf}
\usepackage{lscape}
\usepackage{bm}
\usepackage{cite}
\usepackage{url}
\usepackage{subfigure}
\usepackage{graphicx}
\usepackage{color}
\usepackage{algorithm,algorithmicx,algpseudocode}
\newtheorem{assumption}{Assumption}
\newtheorem{remark}{Remark}
\newtheorem{definition}{Definition}
\newtheorem{theorem}{Theorem}

\newtheorem{lemma}{Lemma}

\definecolor{steelblue}{RGB}{70,130,180}
\definecolor{warmred}{RGB}{244, 104, 65}

\begin{document}
\title{An Online Joint Optimization-Estimation Architecture for Distribution Networks}
\author{Yi Guo, \IEEEmembership{Member, IEEE}, 
Xinyang Zhou, \IEEEmembership{Member, IEEE},
Changhong Zhao, \IEEEmembership{Senior Member, IEEE},\\
Lijun Chen, \IEEEmembership{Member, IEEE},
Gabriela Hug, \IEEEmembership{Senior Member, IEEE}
and Tyler H. Summers, \IEEEmembership{Member, IEEE}

% \thanks{This material is based on work partially supported by funding from US Department of Energy Office of Energy Efficiency and Renewable Energy Solar Energy Technologies Office under contract No. DE-EE-0007998. Funding
% provided by U.S. Department of Energy Office of Energy Efficiency and
% Renewable Energy Solar Energy Technologies Office. The views expressed
% in the article do not necessarily represent the views of the DOE or the U.S.
% Government. The U.S. Government retains and the publisher, by accepting
% the article for publication, acknowledges that the U.S. Government retains a
% nonexclusive, paid-up, irrevocable, worldwide license to publish or reproduce
% the published form of this work, or allow others to do so, for U.S. Government
% purposes.}

\thanks{This material is based on work partially supported by an ETH Z\"{u}rich Postdoctoral Fellowship, the National Science Foundation under grant CMMI-1728605 and Hong Kong RGC Early Career Award No. 24210220. }
\thanks{This work was authored in part by the National Renewable Energy
Laboratory, operated by Alliance for Sustainable Energy, LLC, for the U.S.
Department of Energy (DOE) under Contract No. DE-EE-0007998. Funding
provided by U.S. Department of Energy Office of Energy Efficiency and
Renewable Energy Solar Energy Technologies Office. The views expressed
in the article do not necessarily represent the views of the DOE or the U.S.
Government. The U.S. Government retains and the publisher, by accepting
the article for publication, acknowledges that the U.S. Government retains a
nonexclusive, paid-up, irrevocable, worldwide license to publish or reproduce
the published form of this work, or allow others to do so, for U.S. Government
purposes.}
    \thanks{Y. Guo and G. Hug are with Power Systems Laboratory at ETH Z\"{u}rich, Z\"{u}rich, CH-8092,  Switzerland, email:\{guo, hug\}@eeh.ee.ethz.ch.}
    \thanks{X. Zhou is with Department of Power System Engineering, National Renewable Energy Laboratory, Golden, CO 80401, USA, email: xinyang.zhou@nrel.gov.}
    \thanks{C. Zhao is with Department of Information Engineering, the Chinese University of Hong Kong, HKSAR, China, email: chzhao@ie.cuhk.edu.hk.
    }
    \thanks{L. Chen is with College of Engineering and Applied Science, The University of Colorado, Boulder, CO 80309, USA, email: lijun.chen@colorado.edu.}
    \thanks{T.H. Summers is with the Department
of Mechanical Engineering, The University of Texas at Dallas, Richardson, TX 75080, USA, email: tyler.summers@utdallas.edu.}
\thanks{ \emph{(Corresponding Author: Xinyang Zhou)}.}
}
\maketitle

\begin{abstract}
In this paper, we propose an optimal joint optimization-estimation architecture for distribution networks, which jointly solves the optimal power flow (OPF) problem and static state estimation (SE) problem through an online gradient-based feedback algorithm. The main objective is to enable a fast and timely interaction between the OPF decisions and state estimators with limited sensor measurements. First, convergence and optimality of the proposed algorithm are analytically established. Then, the proposed gradient-based algorithm is modified by introducing statistical information of the inherent estimation and linearization errors for an improved and robust performance of the online OPF decisions. Overall, the proposed method eliminates the traditional separation of operation and monitoring, where optimization and estimation usually operate at distinct layers and different time-scales. Hence, it enables a computationally affordable, efficient and robust online operational framework for distribution networks under time-varying settings.

% This is second-part of our work on optimal control-estimation synthesis in distribution networks. Part I of our paper presented the general idea of this synthesis. 
% %\textit{(Y. Guo, X. Zhou, C. Zhao, G. Hug, L. Chen and T. Summers, ``Running optimal control-estimation synthesis in distribution networks - Part I: Proof-of-Concept, 2021")}. 
% In part II, we synthesize the online optimal power flow (OPF) solvers and online state estimators (SE) to implement a primal-dual gradient algorithm for voltage regulation with limited measurements.
% %into a gradient algorithm for distribution networks. The main objective is to enhance a fast and timely interaction between the optimal controllers of DERs and state estimators with limited sensor measurement for voltage regulation. 
% Convergence and optimality of the proposed algorithm are analytically established from an optimization view. Then, we modify the gradient algorithm by introducing statistical information of the inherent estimation and linearization errors to improve robustness of online control decisions. Overall, the proposed method bridges the traditional gap between control and estimation
% %, where control and estimation operate at distinct layers and different time-scales. 
% and enables a computationally affordable, efficient and robust online operational framework for distribution networks under time-varying settings.
\end{abstract}


\section{Introduction}\label{sec:introduction}
\IEEEPARstart{T}{he} increasing integration of distributed energy resources (DERs) is bringing about unprecedented changes to  distribution networks. As a high penetration level of DERs in distribution networks alters the traditional characters of network states to fast and time-varying, an effective operation of these
networks requires the latest knowledge of network states \cite{kroposki2020autonomous}. Moreover, while operation and monitoring infrastructure in distribution networks are not as widely deployed as in transmission grids, they are considered to be important enabler for an efficient integration of renewable resources. To address the issue of fast-varying system states, we envision a joint OPF-SE architecture that tightly interlinks the optimization and monitoring layers in a fast time-scale to continuously and systemically regulate the outputs of DERs for operational targets.

The tasks of monitoring and optimization of the network have been considered as two separate tasks, e.g. \cite{primadianto2016review,gill2013dynamic,molzahn2017survey}. Prior works focused on real-time OPF methods for distribution networks assume complete availability of network states to implement various optimal operation targets \cite{summers2015stochastic,abido2002optimal,tang2017real,qu2019optimal,wang2020asynchronous}. However, in practice, network states must be estimated with a monitoring system subjected to noisy measurements. To better fuse the real-time state information into OPF solvers, the recently proposed OPF frameworks \cite{colombino2019online,zhou2019online, dall2016optimal,gan2016online,bernstein2019real,bolognani2014distributed,liu2017distributed} leverage the measurement feedback-based online optimization methods to close the loop between the physical measurement information and OPF solvers. These designs take into account real-time data in the OPF decisions to mitigate the effects of inherent disturbances and modelling errors. However, it is not realistic to have real-time physical measurements of system states at every point of a distribution network due to the required communication needs, end-user privacy concerns, and high costs. This calls for a joint design to solve the real-time OPF tasks with an additional state estimation (SE) in-the-loop, which enables OPF decisions to react to real-time information from a limited number of deployed sensors.

% The OPF controllers with SE in-the-loop will react to real-time information of systems states with limited number of deployed sensors.

% To address this challenge, we propose an \textit{online} optimal control and estimation synthesis framework. Our major contributions can be seen from the perspectives of both power and control engineering.
%which promote closely and timely information exchange between control and estimation layers. 
% {\color{red}Overall message: control-estimation structure has been well explored in transmissions systems.}\\
In both transmission and distribution networks, the optimization-estimation architecture, the Supervisory Control and Data Acquisition (SCADA) system, has been developed and implemented to monitor and control the electrical devices for safe and reliable power delivery \cite{thomas2017power}. A variety of functions, including voltage regulation, economic dispatch, automatic generation control and fault detection, can be achieved based on a well-established communication network. 
% The transmission system state estimators \cite{abur2004power}, which utilize the real-time data from the Remote Terminal Units (RTUs) and Phasor Measurement Units (PMUs), provide the best available information to support these control and operational tasks. 
% Prior works that focused on optimal control, state estimation and joint tasks in transmission systems have been highlighted and discussed in \cite{hu2015constrained,el1982performance,alhelou2019deterministic,ameli2018attack,alhelou2019decentralized,liu2021dynamic,fardanesh2001power,aghamolki2015identification,huang2012state,fernando2015novel,pham2015load,trinh2013quasi,wang2016frequency,liu2021stability,monticelli2012state}. 
% These issues of having the optimal control with limited or unknown information have been well studied in recent decades. 
% {\color{red}Overall message: We can not directly borrow the SCADA system to distribution networks. Time-varying, nonlinear, large-scale}\\
However, the current distribution management systems, where the optimization and estimation tasks operate at distinct layers and different time scales, might not be suitable for the future distribution networks with an extensive penetration of DERs. In particular, the time-scale and communication structure required to collect all \emph{node-wise} network states and to solve the optimal dispatch problem may not be consistent with the time-varying distribution-level dynamics. As the most of DERs are connected to distribution networks, it is apparent that the operators have to estimate and optimize the networks faster than ever to cope with renewable variations and yet guarantee system-level optimality. This requests to bypass the current hierarchical setups for a fast control-estimation joint operation, which will be more applicable and practical for an efficient and reliable operation.

To tightly interlink the optimization and estimation tasks in distribution networks, this paper proposes an online gradient-based algorithm to jointly solve the classic OPF and the weighted least squares (WLS) SE problem in parallel, and we demonstrate its performance on the application of voltage regulation in distribution networks. The proposed joint OPF-SE framework allows us to generate optimal online decisions for DERs by utilizing real-time sensor measurements together with an online voltage magnitude estimator. Our preliminary works \cite{guo2020solving,guo2020optimal} proposed a general OPF framework with SE feedback and studied its robustness and convergence performance. The recent paper \cite{picallo2020closing} studied the interaction
between the dynamic state estimation (i.e., Kalman Filter) and
the feedback-based optimization scheme for voltage regulation
in distribution networks.  Here, we significantly extend our previous works in several directions from the perspectives of both power and control engineering:
\begin{itemize}
\item[1)] \textit{\underline{Time Scales:}} One of the significant differences between this submission and our previous works \cite{guo2020solving,guo2020optimal} is that the proposed OPF-SE scheme is posed in different time scales. In our initial work, we solve the OPF problems with SE in the loop by having one gradient update of the OPF problem with the input of a \textbf{fully solved} SE result. In this paper, the time-varying OPF and SE problems are posed in the same time scale, and are jointly solved in parallel. This is done by pursuing each step of the OPF gradient update with only one SE gradient update. This particular design allows online OPF decisions to cope with fast-changing renewable variations while ensuring computational affordability for practical implementation. It also avoids that system changes that occur during the process of the state estimation computations lead to the incorporation of outdated SE information into the OPF problem. To the best of our knowledge, this is the first architecture for distribution networks solving optimization and estimation problems concurrently with a close and timely interaction between two layers, and implemented in an online fashion. 
\item[2)] \textit{\underline{Stochastic Reformulation:}} We leverage the linearized AC power flow equations to formulate the OPF and SE problem but have a state estimation feedback loop to tradeoff between computational efficiency and feasibility. Clearly, the noisy estimation results and power flow linearization errors lead to possible voltage constraint violations. We consider the SE and AC power flow linearization errors in the feedback to reformulate the proposed joint OPF-SE in a stochastic way. Instead of empirically tightening the operational constraints based on the feedback noise as is done in \cite{guo2020optimal}, the statistical information of the feedback noise is used to improve the feasibility and robustness to uncertainties. Namely, a sampling approach based on conditional value-at-risk (CVaR) is applied to quantify the probability of voltage constraint violation. Hence, the proposed online OPF-SE architecture facilitates the close interaction between optimal decisions and the \emph{statistical} knowledge of uncertainties in the SE feedback loop in distribution networks.

% We consider the SE and AC power flow linearization errors in the feedback loop to reformulate the proposed joint OPF-SE in a stochastic way. The statistical information of the feedback noise improves the feasibility performance and robustness to uncertainties. A sampling approach based on conditional value-at-risk (CVaR) is applied to quantify the probability of voltage constraint violation. In this way, the proposed online OPF-SE architecture facilitates the close interaction between optimal decisions and the knowledge of uncertainties in the SE feedback loop.

\item[3)] \textit{\underline{Theoretical Studies:}} More extensive theoretical analyses are presented in this paper compared to our previous works \cite{guo2020solving,guo2020optimal}. Firstly, convergence and optimality are established for the proposed joint OPF-SE algorithm for a particular time step in a static situation. Specifically, we show that jointly solving OPF-SE problems is equivalent to the results from a single optimization problem. Secondly, due to the stochastic reformulation using CVaR, the stochastic OPF-SE problem is not strongly convex on all primal variables, which compromises the convergence of the proposed algorithm. To overcome this challenge, the proposed primal-dual gradient-based approximation leverages Tikhonov regularization terms on both primal and dual variables to facilitate convergence. The optimality difference caused by the primal-dual regularization terms is rigorously characterized. Thirdly, the online tracking performance and convergence analysis are included as well. Note that the proposed framework and associated analysis results considerably broaden the approaches in\cite{guo2020solving,guo2020optimal,dall2016optimal} by establishing convergence and optimality under time-varying conditions with limited state monitoring capabilities. It is also worth to emphasize that we conduct a general error analysis to bound the regularization errors caused by the Tikhonov primal-dual regularization terms. The analysis of regularization errors offer contribution over \cite{koshal2009distributed,koshal2011multiuser}, where the optimality difference caused only by dual variable is characterized. In addition, the error analysis is not only limited to the online joint OPF-SE problem in this paper, but can apply to a general multi-user optimization solved by primal-dual gradient approaches with regularization \cite{koshal2011multiuser}.

% Due to the stochastic reformulation using CVaR, the stochastic OPF-SE problem is not strongly convex on all primal variables, which compromises the convergence of the proposed algorithm. To overcome this challenge, the proposed primal-dual gradient-based approximation leverages Tikhonov regularization terms on both primal and dual variables to facilitate convergence. The optimality difference caused by primal-dual regularization terms is rigorously characterized. The proposed framework considerably broadens the approaches in \cite{dall2016optimal,guo2020solving} by establishing convergence and optimality under time-varying conditions with limited state monitoring capabilities. Meanwhile, the analysis of regularization errors offers contributions over \cite{koshal2009distributed,koshal2011multiuser}, where the optimality difference caused only by dual variables is characterized. Here, we conduct a general error analysis to bound the regularization errors caused by the primal-dual regularization terms. The error analysis is not only limited to the online joint OPF-SE problem in this paper, but can apply to a general multi-user optimization solved by primal-dual gradient approaches with regularization \cite{koshal2011multiuser}.
\end{itemize}
The rest of the paper is organized as follows. Section \ref{sec:system_model} introduces the system model and proposes the joint OPF-SE algorithm. Furthermore, convergence and optimality of the algorithm are established. In Section \ref{sec:stochastic_online}, a stochastic OPF-SE framework is proposed and solved with a regularized primal-dual online gradient method. Section \ref{sec:numerical} presents numerical results, and Section \ref{sec:conclusions} concludes the paper.

\textit{Notation}. The set of real numbers is denoted by $\mathbb{R}$ and the set of non-negative numbers is denoted by $\mathbb{R}_+$. The set of complex numbers is denoted by $\mathbb{C}$. We use $|\cdot|$ to denote the absolute value of a number or the cardinality of a set. Given a matrix $A\in\mathbb{R}^{n \times m}$, $A^\top$ denotes its transpose. We write $A \succeq 0$ ($A \succ 0$) to denote that $A$ is positive semi-definite (definite). For $x\in\mathbb{R}$, the function $[x]_+$ is defined as $[x]_+ := \textrm{max}\{0,x\}$.
For a given column vector $x\in\mathbb{R}^n$, we define $\|x\|_1:= \sum_i |x_i|$ and $\|x\|_2 :=\sqrt{x^\top x}$. Finally, $\nabla_xf(x)$ returns the gradient vector of $f(x)$ with respect to $x \in \mathbb{R}^n$.

%\subsubsection*{Notation} 
%The inner product of two vectors $\mathbf{x},\mathbf{y} \in \mathbb{R}^N$ is denoted by $\mathbf{x}^\top \mathbf{y}$. We use $(\cdot)^\top$ to denote vector or matrix transpose. Let $|\cdot|$ denote the cardinality of a set or the absolute value of a number. Let $\|\mathbf{x}\|_p$ denote the $\bm{\ell}_p$-norm of a vector $\mathbf{x}\in\mathbb{R}^N$, such that $\|\mathbf{x}\|_p := \left(\sum_{i=1}^n |x_i|^p\right)^{1/p}$. For $\mathbf{x}\in\mathbb{R}^N$, define $[\mathbf{x}]_+ :=\max \{\mathbf{x},0\}$ (element-wise). Let $\nabla_\mathbf{x}f(\mathbf{x})/\partial_\mathbf{x}f(\mathbf{x})$ return the gradient/subgradient vector of $f(\mathbf{x}):\mathbb{R}^N \to \mathbf{R}$ with respect to $\mathbf{x}$. We use $N_s$ to represent the number of samples inside the sampling dataset $\hat{\Xi}$. 

\section{Modelling and Primary Problems Setup}\label{sec:system_model}

\subsection{Network Modelling}
Consider a distribution network modelled by a directed and connected graph $\mathcal{G}(\mathcal{N}_0,\mathcal{E})$, where $\mathcal{N}_0:= \mathcal{N}\cup\{0\}$ is the set of all ``buses" or ``nodes" with the substation node 0 and $\mathcal{N}:= \{1,\dots,N\}$. The set $\mathcal{E} \subset \mathcal{N}\times\mathcal{N}$ collects ``links" or ``lines" for all $(i,j) \in \mathcal{E}$. Let $V_{i,t} \in\mathbb{C}$ denote the line-to-ground voltage at node $i\in\mathcal{N}$ at time $t$, where the voltage magnitude is given by $v_{i,t}:=|V_{i,t}|$. Let $p_{i,t}\in\mathbb{R}$ and $q_{i,t}\in\mathbb{R}$ denote the active and reactive power injections of the DER at node $i\in\mathcal{N}$ for all $t>0$. We denote $\mathcal{X}_{i,t}$ as the feasible set of the active and reactive power $p_{i,t}$ and $q_{i,t}$ at node $i\in\mathcal{N}$ for all $t>0$. 
% The feasible set of operating set-points of the DER at node $i\in\mathcal{N}$ is modelled as a convex envelop defined by
% \begin{equation}\nonumber
%     \mathcal{X}_{i,t}:=\bigg\{(p_{i,t},q_{i,t}):p_{i,t}^{\textrm{min}}\leq p_{i,t} \leq p_{i,t}^{\textrm{max}}, p_{i,t}^2 + q_{i,t}^2 \leq (s_{i,t}^{\textrm{max}})^2 \bigg\},
% \end{equation}
% where $s_{i,t}^{\textrm{max}}$ is the apparent power limit of the DER at node $i\in\mathcal{N}$ at time $t$. Let $p_{i,t}^{\textrm{min}}$ and $p_{i,t}^{\textrm{max}}$ denote the lower and upper bounds of the active power set-point of the DER at node $i\in\mathcal{N}$ at time $t$. 
For a PV inverter-based DER, the feasible set $\mathcal{X}_{i,t}$ is constructed by the solar energy availability. For other DERs, such as energy storage systems, small-scale diesel generators and variable frequency drives, the set $\mathcal{X}_{i,t}$ can be appropriately modelled to include their physical capacity limits; see \cite{zhou2019online}. Note that the set $\mathcal{X}_{i,t}$ is convex, closed and bounded for all $i\in\mathcal{N}$ over time $t>0$. For future development, we use $\mathcal{X}_t:=\mathcal{X}_{1,t}\times\ldots\times\mathcal{X}_{N,t}$ to denote the Cartesian product of the feasible sets of all DERs. 
% In addition, we summarize some of the notations used in Section \ref{sec:system_model}, as shown in Table \ref{tab:notation}.

% \begin{table}[ht]
% \centering
% \caption{Notation}
% \begin{tabular}[t]{cl}
% \hline\hline
% $\mathcal{N}$ & set of the nodes/buses\\
% $\mathcal{E}$ & set of the edges/lines\\
% $\mathcal{M}_p$ & set of the nodes with pseudo-measurement (active power)\\
% $\mathcal{M}_q$ & set of the nodes with pseudo-measurement (reactive power)\\
% $\mathcal{M}_v$ & set of the nodes with voltage measurement\\
% $p_{i,t}$ & real power injection at node $i$, $\forall t$\\
% $q_{i,t}$ & reactive power injection at node $i$, $\forall t$\\
% $v_{i,t}$ & voltage magnitude at node $i$, $\forall t$\\
% $\tilde{p}_{i,t}$ & estimation of real power injection at node $i$, $\forall t$\\
% $\tilde{q}_{i,t}$ & estimation of reactive power injection at node $i$, $\forall t$\\
% $\tilde{v}_{i,t}$ & estimation of voltage magnitude at node $i$,$\forall t$\\
% $\hat{p}_{i,t}$ & measurement of real power injection at node $i$, $\forall t$\\
% $\hat{q}_{i,t}$ & measurement of reactive power injection at node $i$, $\forall t$\\
% $\hat{v}_{i,t}$ & measurement of voltage magnitude at node $i$, $\forall t$ \\
% $\mathcal{X}_{i,t}$ & feasible power set of DERs at node $i$, $\forall t$; $\mathcal{X}_t:= \times_{i=1}^N \mathcal{X}_{i,t}$\\
% $[\cdot]_{\mathcal{X}_t}$ & projection onto the feasible set $\mathcal{X}$, $\forall t$\\
% $[\cdot]_{\mathbb{R}_+}$ & projection onto the non-negative orthant\\
% \hline\hline
% \end{tabular}
% \label{tab:notation}
% \end{table}%

The relationships between the voltage, current and net-loads in a distribution network are described by the nonlinear power flow equations (e.g., based on the DistFlow model) as:
\begin{subequations}\label{eq:nonlinear_pf}
    \begin{align}
        P_{ij,t} & = -p_{j,t} + \sum_{k:(j:k)\in\mathcal{E}} P_{jk,t} + r_{ij}\ell_{ij,t}^2,\\
        Q_{ij,t} & = -q_{j,t} + \sum_{k:(j:k)\in\mathcal{E}} Q_{jk,t} + x_{ij}\ell_{ij,t}^2,\\
        v_{j,t}^2 & = v_{i,t}^2 - 2\Big(r_{ij}P_{ij,t} + x_{ij}Q_{ij,t}\Big) + \Big(r_{ij}^2 + x_{ij}^2\Big)\ell_{ij,t}^2,\\
        v_{i,t}^2 \ell_{ij,t}^2 & = P_{ij,t}^2 + Q_{ij,t}^2,
    \end{align}
\end{subequations}
where $P_{ij,t}\in\mathbb{R}$ and $Q_{ij,t}\in\mathbb{R}$ are the real and reactive power flows on line $(i,j)$ at time $t$. We use $k:(j:k) \in \mathcal{E}$ to indicate all the distribution lines $(j,k) \in \mathcal{E}$ connected to bus $j$. The impedance of line $(i,j)\in\mathcal{E}$ is $r_{ij} +\mathbf{j}x_{ij}$. The squared magnitude of the current on line $(i,j)\in\mathcal{E}$ at time $t$ is defined by $\ell_{ij,t}^2 \in \mathbb{R}_{+}$. For convenience, we define vectors $\mathbf{v}_t:=[v_{1,t},\ldots,v_{N,t}]^\top\in\mathbb{R}^N$, $\mathbf{p}_t:=[p_{1,t},\ldots,p_{N,t}]^\top\in\mathbb{R}^N$ and $\mathbf{q}_t:=[q_{1,t},\ldots,q_{N,t}]^\top\in\mathbb{R}^N$.

To formulate a computationally tractable convex optimization problem, we linearize the relationship between voltage magnitudes and nodal power injections as follows: 
\begin{equation}
\label{eq:linear_powerflow}
    \mathbf{v}_t(\mathbf{p}_t,\mathbf{q}_t) = \mathbf{R}\mathbf{p}_t + \mathbf{X}\mathbf{q}_t + \mathbf{v}_0,
\end{equation}
where the parameters $\mathbf{R}\in\mathbb{R}^{N\times N}$, $\mathbf{X}\in\mathbb{R}^{N\times N}$ and $\mathbf{v}_0\in\mathbb{R}^{N}$ can be attained from various linearization methods, e.g., \cite{bolognani2015fast,gan2016online,baran1989network}. For the rest of this paper, we consider the voltage magnitude projection $\mathbf{v}_t(\cdot)$ as a fixed linearization with time-invariant matrices $\mathbf{R}$ and $\mathbf{X}$ for simplicity, although it is straightforward to extend these two matrices to time-varying linearized models. With the model above, we formulate a time-varying OPF problem for voltage regulation and a static WLS voltage estimation problem at each given time $t>0$. An online gradient algorithm with feedback is developed to solve these two optimization problems in parallel.

\subsection{Joint OPF-SE via Primal-Dual Gradient Feedback}
For the application of real-time voltage regulation, we introduce a time-varying OPF problem and a WLS-based SE problem to attain optimal set-points of DERs at given time $t>0$. In the subsequent sections, we first give the formulations for the OPF and the SE problems. Then, a primal-dual gradient algorithm is proposed to %solve the two optimization problems in parallel with a feedback loop to simultaneously 
utilize the state estimation results in the OPF solution at each iteration. 
%Convergence of the proposed algorithm will be established. 
We next prove the convergence of the proposed algorithm by showing that the equilibrium point of the proposed algorithm is equivalent to the saddle-point dynamics of a single OPF-SE synthesis optimization formulation.
%This optimization-based model allows us to further develop an online stochastic OPF-SE synthesis algorithm in Section \ref{sec:stochastic_online}.

\subsubsection{OPF Problem}
Consider a time-varying OPF problem $(\bm{\mathcal{P}_t^1})$ for voltage regulation:
\begin{subequations}\label{eq:opf}
\begin{eqnarray}
 (\bm{\mathcal{P}^1_t}) & \underset{\mathbf{u}_t}{\min}  & C_t^{\textrm{OPF}}(
\mathbf{u}_t),\\%
&\text{subject to}& \mathbf{r}(\mathbf{v}_t(\mathbf{u}_t)) \leq 0, \label{eq:opf_voltreg_upper_lower}\\
&& \mathbf{u}_t\in\mathcal{X}_t,
\label{eq:opf_feasible_set}
\end{eqnarray}
\end{subequations}
where $\mathbf{u}_t:=[\mathbf{p}_t^\top,\mathbf{q}_t^\top]^\top\in\mathbb{R}^{2N}$ denotes the power set-points of DERs at time $t$. The voltage constraints are given in a compact form, such that $\mathbf{r}(\mathbf{v}_t(\mathbf{u}_t)):= [(\mathbf{v}_t(\mathbf{u}_t) - \mathbf{v}^{\textrm{max}})^\top, (\mathbf{v}^{\textrm{min}} - \mathbf{v}_t(\mathbf{u}_t))^\top]^\top \in \mathbb{R}^{2N}$. The voltage constraint in \eqref{eq:opf_voltreg_upper_lower} utilizes the linearized AC power flow \eqref{eq:linear_powerflow}, where the lower and upper limits are denoted by $\mathbf{v}^{\textrm{min}}\in\mathbb{R}^N$ and $\mathbf{v}^{\textrm{max}}\in\mathbb{R}^N$, respectively. The set-points of DERs at time $t$ are subjected to the convex and compact feasible set $\mathcal{X}_t$. The OPF function $C_t^{\textrm{OPF}}(\cdot):\mathbb{R}^{2N}\to\mathbb{R}$ is a generic time-varying cost objective at time step $t$, capturing the costs of the system operator, e.g., the costs of deviations of the power flow into the substation from its reference values and/or the costs of power production by DERs, including generation costs, ramping costs, the renewable curtailment penalty, the auxiliary service expense and the reactive power compensation.

\subsubsection{Motivation for Involving SE}
Problem~\eqref{eq:opf} is typically solved assuming that all network voltages $\{\mathbf{v}_t\}$ and set-points $\{\mathbf{u}_t\}$ of DERs are available in real-time. However, there is generally a lack of reliable measurements and timely communications in practical distribution networks, hindering effective implementation of conventional OPF approaches which usually assume the availability of all system states. Therefore, a major challenge for solving \eqref{eq:opf} lies in gathering real-time system state information such as net-loads and voltages that can be integrated into an OPF solver. Although the system states of a distribution network are not fully measurable in practice, the distribution network can be fully observable by a well-posed SE problem with pseudo-measurements
% \footnote{Due to the lack of real-time measurements and the stochasticity nature of power nodal injections in distribution system state estimation, the nodal power injections are measured by their nominal load-pattern (i.e., the real value plus zero-mean random deviations), so-called pseudo-measurement, whose information is derived from the past records of load behaviors \cite{schweppe1970power}.} 
for all nodal injections and a limited number of voltage measurements \cite{schweppe1970power}. Hence,
we will tackle this challenge by fusing %time-varying state estimation with a limited number of sensor measurements. This allows the real-time OPF controller to respond to timely information and update control decisions despite nodes without measurement in the grid. 
the time-varying SE problem $(\bm{\mathcal{P}_t^2})$ with the OPF problem $(\bm{\mathcal{P}_t^1})$. 
%Next, we will show how to integrate two gradient-based approaches for solving $(\bm{\mathcal{P}_t^1})$ and $(\bm{\mathcal{P}_t^2})$, to develop one gradient-based OPF-SE synthesis algorithm with feedback loop. 
%for solving these two problems in parallel.

\begin{remark}[Pseudo-measurements]  Due to the lack of real-time measurements and the stochastic nature of power nodal injections in distribution system state estimation, the nodal power injections are measured by their nominal load-pattern (i.e., the real value plus zero-mean random deviations), so-called pseudo-measurement, whose information is derived from the past records of load behaviors \cite{schweppe1970power}. As shown in Lemma \ref{lem:SE_observability} below, having pseudo-measurements of power injections contributes to establishing full observability of the static SE problem. In principle, we can instead use the real-time values of power injections, which however requires an efficient and fast communication structure and a sufficient monitoring system having sensors at all nodes. This requirement currently cannot be satisfied in a large-scale distribution network. Hence, in this paper, we use the historical data of all nodal power injections as pseudo-measurements to tradeoff the accuracy of estimation results and the cost of measurements, which has been observed to be efficient in a static SE problem \cite{dvzafic2013real}. \label{rmk:pseudo-measurements}
\end{remark}

% \begin{remark}[Pseudo-measurements] Due to the lack of real-time measurements and the stochasticity nature of power nodal injections in distribution system state estimation, the nodal power injections are measured by their nominal load-pattern (i.e., the real value plus zero-mean random deviations), so-called pseudo-measurement, whose information is derived from the past records of load behaviors \cite{schweppe1970power}.  As shown in Lemma 1 below, having pseudo-measurements of power injections contributes to the full observability of the static SE problem. In principle, we can access the real-time values of power injections, which requires an efficient and fast communication structure and a sufficient monitoring system having sensors at all nodes. This requirement currently can not be satisfied in a large-scale distribution network. In this paper, we use the historical data of all nodal power injections as pseudo-measurements to tradeoff the accuracy of estimation results and the cost of measurements, which has been observed to be efficient in a static SE problem \cite{dvzafic2013real}. \label{rmk:pseudo-measurements}
% \end{remark}\color{black}

\subsubsection{SE Problem}
System states are a set of variables that can determine the behavior of the entire system, i.e., the power flow equations here. For a distribution system, either voltage phasors of all nodes or nodal real and reactive power injections for all nodes can be chosen as system states. Consider a time-varying measurement model with the true system state vector $\mathbf{z}_t \in \mathbb{R}^{2N}$ and the measurement vector $\mathbf{y}_t \in \mathbb{R}^{m}$ at time $t$. The measurement function is defined as $\mathbf{h}_t(\cdot):\mathbb{R}^{2N}\to\mathbb{R}^{m}$. The measurement noise $\bm{\xi}_t^{\textrm{SE}}\in\mathbb{R}^{m}$ follows a normal probability distribution with zero mean and covariance matrix $\bm{\Sigma}_t\in\mathbb{R}^{m \times m}$,
\begin{equation}\label{eq:state_measurement}
    \mathbf{y}_t = \mathbf{h}_t(\mathbf{z}_t) + \bm{\xi}_t^{\textrm{SE}}.
\end{equation}
Using $\mathbf{W}_t:= (\bm{\Sigma}_t)^{-1}$ as the weighting matrix, a time-varying WLS SE problem is formulated as:
\begin{equation}\nonumber
    \min_{\tilde{\mathbf{z}}_t}~~~ \frac{1}{2}(\mathbf{y}_t-\mathbf{h}_t(\tilde{\mathbf{z}}_t))^\top \mathbf{W}_t(\mathbf{y}_t-\mathbf{h}_t(\tilde{\mathbf{z}}_t)),
\end{equation}
where $\tilde{\mathbf{z}}_t$ denotes the estimation of the true state $\mathbf{z}_t$. In this work,  considering the optimal voltage regulation problem $\bm{\mathcal{P}^1_t}$,  we adopt voltage magnitude measurement at selected nodes as the real-time measurement, together with pseudo-measurement for all load nodes, and
construct the following WLS problem for SE:
%for voltage measurement in  $(\bm{\mathcal{P}_t^2})$
\begin{subequations}\label{eq:se}
\begin{eqnarray}
(\bm{\mathcal{P}^2_t})
& \underset{\tilde{\mathbf{z}}_t,\tilde{\mathbf{v}}_t}{\min}& C^{\textrm{SE}}_t\left(\tilde{\mathbf{z}}_t,\tilde{\mathbf{v}}_t \right), \label{eq:se_obj}\\
& \textrm{subject to} & \tilde{\mathbf{v}}_t= \mathbf{v}_t(\tilde{\mathbf{z}}_t),\label{eq:se_linearization_v}
% & & \tilde{\mathbf{z}}_t \in \mathcal{X}_t^{\textrm{SE}},
\end{eqnarray}
\end{subequations}
where $\tilde{\mathbf{z}}_t=[\tilde{\mathbf{p}}_t^{\top}, \tilde{\mathbf{q}}_t^{\top}]^{\top}$ collects the estimated active and reactive power injections for all nodes with $\tilde{\mathbf{p}}_t := [\tilde{p}_{1,t},\ldots,\tilde{p}_{N,t}]^\top \in \mathbb{R}^N$ and $\tilde{\mathbf{q}}_t:= [\tilde{q}_{1,t},\ldots,\tilde{q}_{N,t}]^\top \in \mathbb{R}^N$, the estimated voltage magnitudes are denoted by $\tilde{\mathbf{v}}_t:= [\tilde{v}_{1,t},\ldots,\tilde{v}_{N,t}]^\top \in \mathbb{R}^N$. 
% \color{blue}We use $\mathcal{X}_t^{\textrm{SE}}$ to denote the feasible set of estimation variables at time $t>0$. \color{black} 
The objective function $C^{\textrm{SE}}_t:=\sum_{i\in\mathcal{M}_p}\frac{\left(\hat{p}_{i,t}-\tilde{p}_{i,t}\right)^2 }{2(\sigma^p_{i,t})^2}+\sum_{i\in\mathcal{M}_q}\frac{\left(\hat{q}_{i,t}-\tilde{q}_{i,t} \right)^2}{2(\sigma^q_{i,t})^2} + \sum_{i\in\mathcal{M}_v}\frac{\left(\hat{v}_{i,t}-\tilde{v}_{i,t}\right)^2}{2(\sigma^v_{i,t})^2}$ features the weighted sum of all costs of measurements based on their respective accuracy, where the sets $\mathcal{M}_p$, $\mathcal{M}_q$ and $\mathcal{M}_v$ contain the nodes with pseudo-measurements of the active and reactive power injections, and voltage magnitude measurements, respectively. The vectors $\hat{\mathbf{p}}_t :=\{\hat{p}_{i,t} | \forall i \in \mathcal{M}_p\}$, $\hat{\mathbf{q}}_t :=\{\hat{q}_{i,t} | \forall i \in \mathcal{M}_q\}$ collect the pseudo-measurements of the loads and $\hat{\mathbf{v}}_t := \{\hat{v}_{i,t} | \forall i\in\mathcal{M}_v\}$ gathers the sensor measurements of voltage magnitudes. The noisy voltage magnitude measurements $\hat{v}_i, \forall i\in\mathcal{M}_v$ are assumed to be attained from voltage magnitude measurements with relatively high accuracy. The pseudo-measurements of the active power injections $\hat{p}_{i,t}, \forall i\in\mathcal{M}_p$ and the reactive power injections $\hat{q}_{i,t},
\forall i\in\mathcal{M}_q$ are attained from the historical data assuming large variations. The standard deviations of the measurement errors are denoted by $\sigma_{i,t}^p$, $\sigma_{i,t}^q$ and $\sigma_{i,t}^v$ for active power, reactive power and voltage magnitude, respectively. We assume that measurement errors are independent.
% To guarantee the full observability of the state estimation problem \eqref{eq:se}, the sensor measurements of voltage magnitude and the pseudo-measurements of power injections will be taken into account in the rest of this paper.
%We specify the state vector as $\mathbf{z}_t:= [\tilde{\mathbf{p}}_t^\top,\tilde{\mathbf{q}}_t^\top]^\top$ and measurement vector $\mathbf{y}_t:= [\hat{\mathbf{p}}_t^\top,\hat{\mathbf{q}}_t^\top,\hat{\mathbf{v}}_t^\top]^\top$. 
The estimation variable $\tilde{\mathbf{z}}_t$ is subjected to a convex and compact feasible set $\mathcal{X}_t$. Note that the states of a distribution network are uniquely determined by \eqref{eq:nonlinear_pf} given $\mathbf{z}_t$ at any time $t$. We leverage the linearized AC power flow model \eqref{eq:linear_powerflow} to determine the voltage magnitude as $\tilde{\mathbf{v}}_t= \mathbf{v}_t(\mathbf{z}_t)$.

\begin{definition}[Full Observability \cite{wu1985network,abur2004power}] A state-to-output system $\mathbf{y} = \mathbf{h}(\mathbf{z})$ is fully observable\footnote{The observability of the time-varying WLS SE problem should be distinguished from observability of linear dynamical systems. Here, we limit the definition of observability to power system static
state estimation problems \cite{schweppe1970power}.} if $\mathbf{z} = 0$ is the only solution for $\mathbf{h}(\mathbf{z}) = 0$. This condition allows a unique solution to \eqref{eq:se}.
\end{definition}

% \begin{definition}[Full Observability \cite{wu1985network,abur2004power}] A state-to-output system $\mathbf{y} = \mathbf{h}(\mathbf{z})$ is fully observable\footnote{\color{blue}The observability of the time-varying WLS SE problem should be distinguished from observability of linear dynamical systems. Here, we limit the definition of observability to power system static
% state estimation problems \cite{schweppe1970power}.} if $\mathbf{z} = 0$ is the only solution for $\mathbf{h}(\mathbf{z}) = 0$, which allows a unique solution to \eqref{eq:se}.
% \end{definition}

\begin{lemma}[Sufficient Condition for Full Observability \cite{zhou2019graident}] A sufficient condition for the distribution network $\mathcal{G}$ to be fully observable by \eqref{eq:se} is $\mathcal{M}_p = \mathcal{M}_q = \mathcal{N_L}$, where the set $\mathcal{N}_L$ collects all nodes with power injections.\label{lem:SE_observability}
\end{lemma}

\noindent For the rest of this paper, we assume to have pseudo-measurements of all power injections to guarantee full observability. This setting is easy to be satisfied in practice and has been effective to obtain high accuracy estimation results in \cite{dvzafic2013real}. For more observability analysis results, numerical studies and discussions, we refer the reader to our previous works \cite{zhou2019graident}. Optimally deploying the voltage sensors for accuracy improvement is out of the scope of this paper but it is an interesting topic to investigate in future work. 

% \noindent For the rest of this paper, we assume to have pseudo-measurements of all power injections to guarantee full observability. This setting is easier to be satisfied in practice and has been observed to be effective in \cite{dvzafic2013real}, which leads to high accuracy of the estimation results. For more observability analysis results, numerical studies and discussions, we refer readers to our previous works \cite{zhou2019graident}. Optimally deploying the voltage sensors for accuracy improvement is out of the scope of this paper but it will be an interesting topic to investigate for the future works.
% \color{black}
% \begin{assumption}\label{ass:SE_observability}
% For every $t>0$, system \eqref{eq:state_measurement} is fully observable
% \end{assumption}
%To enable a computationally efficient implementation of time-varying OPF with timely SE-feedback in distribution networks. We utilize the primal-dual gradient-based algorithms to solve the voltage regulation problem
%in \eqref{eq:opf} and voltage estimation problem \eqref{eq:se} in parallel. Every gradient step for solving set-points of DERs \eqref{eq:opf} constantly relies on the each gradient update for solving the voltage estimation \eqref{eq:se}. 

\subsubsection{Joint OPF-SE Algorithm}
To solve problems $(\bm{\mathcal{P}^1_t})$, 
% and $(\bm{\mathcal{P}^2_t})$ in parallel via a gradient approach,
we consider the regularized Lagrangian of \eqref{eq:opf}:
\begin{equation}\label{eq:l_opf}
\begin{aligned}
    \mathcal{L}^{\textrm{OPF}}_t(\mathbf{u}_t,\bm{\mu}_t) & = C^{\textrm{OPF}}_t(\mathbf{u}_t) + \bm{\mu}_t^\top\mathbf{r}(\mathbf{v}_t(\mathbf{u}_t)) - \frac{\phi}{2}\|\bm{\mu}_t\|_2^2,
\end{aligned}
\end{equation}
where $\bm{\mu}_t \in \mathbb{R}^{2N}$ is the vector of Langrange multipliers associated with constraint \eqref{eq:opf_voltreg_upper_lower}. The Lagrangian \eqref{eq:l_opf} includes a Tikhonov regularization term $-\frac{\phi}{2}\|\bm{\mu}_t\|_2^2$ with a small constant $\phi >0$. This regularization term facilitates the convergence of the primal-dual algorithm to the solution of the saddle-point problem:
\begin{equation}\label{eq:saddle_point_static_opf}
\max_{\bm{\mu}_t\in\mathbb{R}^{2N}} \min_{\mathbf{u}_t\in\mathcal{X}_t} \mathcal{L}^{\textrm{OPF}}_t(\mathbf{u}_t,\bm{\mu}_t),
\end{equation}
which is an approximate optimal solution of the original problem \eqref{eq:opf}. The difference between the solutions of the original problem and the regularized problem was characterized in \cite{koshal2011multiuser}. We first use the primal-dual gradient method to solve the saddle-point problem \eqref{eq:saddle_point_static_opf} as the existing results \cite{zhou2019accelerated,dall2016optimal}:
\begin{subequations}\label{eq:algorithm_opf_t}
\begin{align}
\mathbf{u}_t^{k+1} & =  \Bigg[\mathbf{u}_t^k - \epsilon\bigg( \nabla_{\mathbf{u}}C_t^{\textrm{OPF}}(\mathbf{u}_t^k) + \nabla_{\mathbf{u}}\mathbf{r}(\mathbf{v}_t(\mathbf{\mathbf{u}}_t^k))^\top\bm{\mu}_t^k\bigg)\Bigg]_{\mathcal{X}_t},\label{eq:equilibrium_p_opf}\\
% \mathbf{z}_t^{k+1} & =  \mathbf{z}_t^k - \epsilon \nabla_{\mathbf{z}}C_t^{\textrm{SE}}(\mathbf{z}_t^k),\label{eq:equilibrium_p_se} \\
%\mathbf{v}_t^{k+1} & = \mathbf{v}_t(\mathbf{z}_t^k),\label{eq:voltage_feedback_algorithm_feedback}\\
\bm{\mu}_t^{k+1} & = \Bigg[ \bm{\mu}_t^k + \epsilon\bigg(\mathbf{r}(\hat{\mathbf{v}}_t^k) - \phi\bm{\mu}_t^k\bigg) \Bigg]_{\mathbb{R}_+},\label{eq:static_equilibrium_mu_lower_opf}\\
\hat{\mathbf{v}}_t^{k+1} & \longleftarrow \textrm{based on the voltage magnitude sensors}. \label{eq:voltage_measurement_feedback}
\end{align}
\end{subequations}
where the operator $[\cdot]_{\mathcal{X}_t}$ projects onto the feasible region $\mathcal{X}_t$ and the operator $[\cdot]_{\mathbb{R}_+}$ projects onto the nonnegative orthant. This classic feedback-based gradient approach assumes that all voltage magnitudes $\hat{\mathbf{v}}_t$ are measurable \eqref{eq:voltage_measurement_feedback} for dual updates at every iteration \eqref{eq:static_equilibrium_mu_lower_opf}. However, in practice, this design, particularly \eqref{eq:voltage_measurement_feedback}, heavily depends on reliable real-time measurement devices, which is not well-deployed in the current distribution networks. To tackle this challenge for iteratively solving \eqref{eq:saddle_point_static_opf} in a fast time-scale, we introduce a SE feedback loop as shown in \eqref{eq:equilibrium_p_se} below. The gradient updates of the SE problem \eqref{eq:se} timely feedback the estimates of voltage magnitudes $\mathbf{v_t}(\mathbf{z}_t^k)$ into dual updates in parallel with every OPF updates with a limited number of sensor measurements. The proposed jointly OPF-SE algorithm is shown as follow:
\begin{subequations}\label{eq:algorithm_joint_t}
\begin{align}
\mathbf{u}_t^{k+1} & =  \Bigg[\mathbf{u}_t^k - \epsilon\bigg( \nabla_{\mathbf{u}}C_t^{\textrm{OPF}}(\mathbf{u}_t^k) + \nabla_{\mathbf{u}}\mathbf{r}(\mathbf{v}_t(\mathbf{\mathbf{u}}_t^k))^\top\bm{\mu}_t^k\bigg)\Bigg]_{\mathcal{X}_t},\label{eq:equilibrium_p}\\
%\mathbf{v}_t^{k+1} & = \mathbf{v}_t(\mathbf{z}_t^k),\label{eq:voltage_feedback_algorithm_feedback}\\
\bm{\mu}_t^{k+1} & = \Bigg[ \bm{\mu}_t^k + \epsilon\bigg(\mathbf{r}(\mathbf{v}_t(\mathbf{z}_t^k)) - \phi\bm{\mu}_t^k\bigg) \Bigg]_{\mathbb{R}_+},\label{eq:static_equilibrium_mu_lower}\\
\mathbf{z}_t^{k+1} & =  \mathbf{z}_t^k - \epsilon \nabla_{\mathbf{z}}C_t^{\textrm{SE}}(\mathbf{z}_t^k).\label{eq:equilibrium_p_se} 
\end{align}
\end{subequations}

Using \eqref{eq:se_linearization_v}, we can rewrite the state estimation objective $C^{\textrm{SE}}_t\left(\mathbf{z}_t,\tilde{\mathbf{v}}_t \right)$ in \eqref{eq:se_obj} as $C^{\textrm{SE}}_t\left(\mathbf{z}_t, \mathbf{v}_t(\mathbf{z}_t)\right)$, abbreviated as $C^{\textrm{SE}}_t\left(\mathbf{z}_t\right)$ in \eqref{eq:equilibrium_p_se}.
%The gradient steps are implemented onto the primal SE variables associated with projection onto the feasible set $\mathcal{X}_t$. 
In each iteration $k$, the OPF commands---the primal variables---are updated once on the end-user side \eqref{eq:equilibrium_p} and the dual variables are updated in \eqref{eq:static_equilibrium_mu_lower} based on the updated estimated voltages $\mathbf{v}_t(\mathbf{z}_t^k)$. Next, the state estimations are updated on the operator side \eqref{eq:equilibrium_p_se}. %{\color{red}The voltages are updated according to the newest state estimations \eqref{eq:voltage_feedback_algorithm_feedback}.} 
By this design, we iteratively connect optimization and estimation tasks in a loop for every update, %Finally, the updated dual variables will be past the end-users to pursuit the control decisions of DERs for the next iteration. 
 %The gradient updates of the proposed joint OPF-SE synthesis algorithm are demonstrated in 
as illustrated in Fig.~\ref{fig:OPF_SE_diagram_1}. Key to this end is to notice that the iterative updates \eqref{eq:equilibrium_p_se} and feedback loop \eqref{eq:static_equilibrium_mu_lower} of state information enable the OPF updates to have the timely tracking of the fast-changing network states at every iteration. This point is particularly important because SE updates \eqref{eq:equilibrium_p_se} can be conceivably performed at a fast time-scale comparing to completely solving the time-varying SE problem \eqref{eq:se}, where the estimation results might be suboptimal once the computation is complete due to the network response from OPF iterative updates.

\begin{figure}[!htbp]
\centering
\includegraphics[width=3.5in]{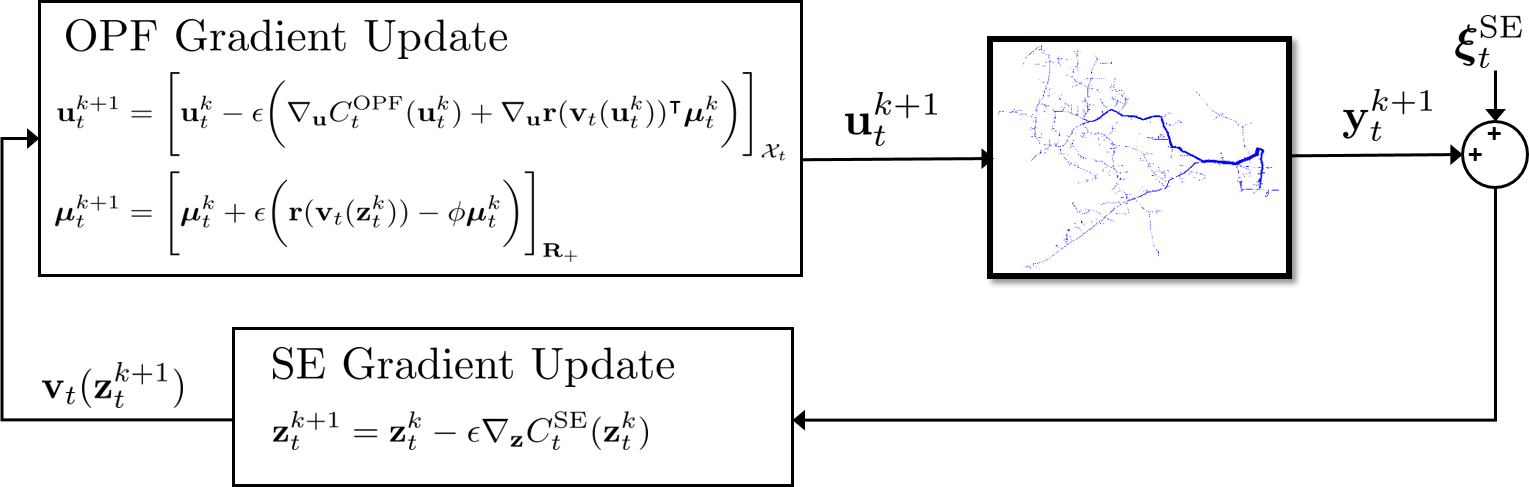}
\caption{Online joint optimization-estimation
architecture in distribution networks.}
\label{fig:OPF_SE_diagram_1}
\end{figure}

% \begin{algorithm}
%     \caption{(Joint OPF-SE Synthesis Algorithm)}
%     \begin{algorithmic}[1]\label{alg:static_algorithm}
%         \Require[S0] Netload control variables initialization $\mathbf{u}^0_t$ and $\bm{\mu}_t^0$ on the end-users side. Estimation variables initialization $\mathbf{z}^0_t$, netloads and voltage measurement $\mathbf{y}^0_t$ and associated weighted matrix $\mathbf{W}_t$.
%         \For{$k = 0:K$}
%             \State[S1] End-users update control variables \eqref{eq:equilibrium_p}.
%             \State[S2] Network operator measures voltage magnitude $\hat{v}_{i,t}, \forall i\in\mathcal{M}_v$.
%             \State[S3] Network operator updates estimation variables \eqref{eq:equilibrium_p_se}.
%             \State[S4] Network operator updates the estimated voltage profile based on linear power-flow equation \eqref{eq:voltage_feedback_algorithm_feedback}.
%             \State[S5] Network operator utilizes the estimation feedback \eqref{eq:voltage_feedback_algorithm_feedback} to update the dual variables \eqref{eq:static_equilibrium_mu_lower}.
%             \State[S6] Network operator transmits the dual variables $\bm{\mu}_t^{k+1}$ to end-users.
%         \EndFor
% \end{algorithmic}
% \label{alg:static_algorithm}
% \end{algorithm}

\subsection{Convergence Analysis}
We now show the convergence of the joint OPF-SE Algorithm \eqref{eq:algorithm_joint_t} under the following assumptions:
\begin{assumption}[Slater's Condition]
\label{ass:slater_deterministic} For every $t>0$, there exists a strictly feasible point $\mathbf{u}_t \in \mathcal{X}_t$, so that 
\begin{equation}\nonumber
    \mathbf{r}(\mathbf{v}_t(\mathbf{u}_t)) \leq 0.
\end{equation}
\end{assumption}
\noindent The Slater's condition here guarantees strong duality and the feasibility of $\bm{(\mathcal{P}_t^1)}$ given any particular time step $t>0$. This implies that the OPF problem \eqref{eq:opf} is well-posed with proper prescribed voltage limits, which holds for a well-design distribution network. Note that the refined Slater's condition is applied here for the above affine voltage constraints without strict inequality.

\begin{assumption}\label{ass:C_OPF}
For every $t>0$, the objective function $C_t^{\textrm{OPF}}(\cdot)$ is continuously differentiable and strongly convex. %as functions of $\mathbf{u}_t$. 
Its first order derivative is bounded 
%within their operation regions indicated by 
on the set $\mathcal{X}_t$.
\end{assumption}
\noindent Assumption \ref{ass:C_OPF} holds for the DERs with convex quadratic cost functions in practice, so that their first order derivatives are naturally bounded in the feasible set $\mathcal{X}_t$ due to the physical limits.
\noindent The primal-dual gradient updates \eqref{eq:algorithm_joint_t} are re-written as:
\begin{equation}\nonumber
\begin{bmatrix}
\mathbf{u}_t^{k+1}\\
\bm{\mu}_t^{k+1}\\
\mathbf{z}_t^{k+1}
\end{bmatrix} = \begin{bmatrix}
\begin{bmatrix}
\mathbf{u}_t^k\\
\bm{\mu}_t^k\\
\mathbf{z}_t^k
\end{bmatrix} - 
\epsilon\begin{bmatrix}
\nabla_{\mathbf{u}}\mathcal{L}_t^{\textrm{OPF}}(\mathbf{u}_t^k,\bm{\mu}_t^k)\\
-\nabla_{\bm{\mu}}\mathcal{L}_t^{\textrm{OPF}}(\mathbf{v}_t(\mathbf{z}_t^k),\bm{\mu}_t^k)\\
\nabla_{\mathbf{z}}C_t^{\textrm{SE}}(\mathbf{z}_t^k)
\end{bmatrix}
\end{bmatrix}_{\mathcal{X}_t\times\mathbb{R}_+\times\mathbb{R}}.
\end{equation}
%where $\mathbb{R}_+$ denotes the feasible non-negative orthant for the dual variables. 
For convenience, we define the gradient operator:
\begin{equation}\label{eq:F_operator}
    \mathcal{F}_t(\mathbf{u}_t,\bm{\mu}_t,\mathbf{z}_t):= \begin{bmatrix}
\nabla_{\mathbf{u}}\mathcal{L}^{\textrm{OPF}}_t(\mathbf{u}_t,\bm{\mu}_t)\\
-\nabla_{\bm{\mu}}\mathcal{L}^{\textrm{OPF}}_t(\mathbf{v}_t(\mathbf{z}_t),\bm{\mu}_t)\\
\nabla_{\mathbf{z}}C^{\textrm{SE}}_t(\mathbf{z}_t)
\end{bmatrix}.
\end{equation}

\begin{lemma}\label{lemma:strongly_monotone}
The gradient operator $\mathcal{F}_t(\mathbf{u}_t,\bm{\mu}_t,\mathbf{z}_t)$ is strongly monotone for all $t>0$.
\end{lemma}
\begin{proof}
See Appendix \ref{appendix_lemma_monotone}.
\end{proof}
By Lemma \ref{lemma:strongly_monotone}, there exists some constants $\widehat{M}>0$, such that for any $\hat{\mathbf{e}}_t:=[\mathbf{u}_t^\top,\bm{\mu}_t^\top,\mathbf{z}_t^\top]^\top$ and $\hat{\mathbf{e}}_t':=[(\mathbf{u}_t')^\top,(\bm{\mu}_t')^\top,(\mathbf{z}_t')^\top]^\top$: 
%for any $\mathbf{\hat{e}}_t, \mathbf{\hat{e}}_t' \in \mathcal{X}_t\times\mathcal{X}_t\times\mathbb{R}_+$, 
\begin{equation}\label{eq:strongly_monotone_F}
    \left(\mathcal{F}_t(\hat{\mathbf{e}}_t) - \mathcal{F}_t(\hat{\mathbf{e}}_t' )\right)^\top  \left(\hat{\mathbf{e}}_t - \hat{\mathbf{e}}_t'\right) \geq \widehat{M}\|\hat{\mathbf{e}}_t - \hat{\mathbf{e}}_t'\|_2^2.
\end{equation}
In addition, the operator $\mathcal{F}_t(\cdot)$ is Lipschitz continuous with some constants $\widehat{L}>0$
under Assumption \ref{ass:C_OPF}, such that for any $\mathbf{\hat{e}}_t, \mathbf{\hat{e}}_t' \in \mathcal{X}_t\times\mathcal{X}_t\times\mathbb{R}_+$, we have:
\begin{equation}\label{eq:Lipschitz_continuous_F}
    \|\mathcal{F}_t(\hat{\mathbf{e}}_t) - \mathcal{F}_t(\hat{\mathbf{e}}_t')\|_2^2 \leq \widehat{L}^2\|\hat{\mathbf{e}}_t - \hat{\mathbf{e}}_t'\|_2^2.
\end{equation}
Given the results above, the convergence of  \eqref{eq:algorithm_joint_t} can be established with a small enough step-size for gradient updates.
\begin{lemma}
\label{lem:stepsize}
Suppose Assumptions \ref{ass:slater_deterministic} and \ref{ass:C_OPF} hold. For any step-size that satisfies:
\begin{equation}
    0 < \epsilon < 2\widehat{M}/\widehat{L}^2,
\end{equation}
the proposed primal-dual gradient algorithm \eqref{eq:algorithm_joint_t}  exponentially converges to the unique saddle point $(\mathbf{u}_t^*,\bm{\mu}_t^*)$ of the saddle point problem \eqref{eq:saddle_point_static_opf} and the unique optimal $\mathbf{z}_t^*$ of the static state estimation problem \eqref{eq:se} for any given time $t>0$.
\end{lemma}
\begin{proof}
% \color{red}
The proof follows closely along the line of the analysis in \cite{zhou2019accelerated,guo2020solving} by noting the strongly monotone \eqref{eq:strongly_monotone_F} and the Lipschitz continuity \eqref{eq:Lipschitz_continuous_F} properties of the operator $\mathcal{F}_t(
\cdot)$. In particular, the additionally parallel SE gradient update $\nabla_{\mathbf{z}}C_t^{\textrm{SE}}(\mathbf{z}_t^k)$ in \eqref{eq:F_operator} compared to previous operators in \cite{zhou2019accelerated,guo2020solving} does not change the properties of Lipschitz continuity and strong monotonicity.
\end{proof}
\noindent Note that the convergence and optimality results of the proposed joint OPF-SE algorithm in this section hold for a particular time step in a static situation. In the following section, we will explore its convergence performance under a time-varying setting.

\subsection{Optimality Analysis}\label{sec:reverse_engineering}
We now take a new perspective to interpret the dynamics \eqref{eq:algorithm_joint_t} as a primal-dual algorithm to solve an optimization problem that unifies OPF and SE. 
%This allows us to leverage the optimization model to design a new stochastic joint OPF-SE scheme and further discuss its online implementation in Section \ref{sec:stochastic_online}.
To be specific, the trajectory of  \eqref{eq:algorithm_joint_t} 
%can be seen as a gradient sequence to 
approaches an optimal solution to the following problem: 
\begin{subequations}\label{eq:opf_se}
\begin{eqnarray}
(\bm{\mathcal{P}^3_t}) &\underset{ \mathbf{u}_t,\mathbf{z}_t,\mathbf{v}_t}{\min} & C^{\textrm{OPF}}_t(\mathbf{u}_t) +  C^{\textrm{SE}}_t(\mathbf{z}_t,\mathbf{v}_t(\mathbf{z}_t)),\\
&\textrm{subject to}&\mathbf{r}(\mathbf{v}_t(\mathbf{u}_t)) \leq 0,\label{eq:p3_voltage_constraint_upper}\\
& & \mathbf{v}_t(\mathbf{u}_t) = \mathbf{v}_t(\mathbf{z}_t),\label{eq:p3_se}\\
& & \mathbf{u}_t\in\mathcal{X}_t.
\end{eqnarray}
\end{subequations}
The objective function of $(\bm{\mathcal{P}_t^3})$ includes operational costs  $C_t^\textrm{OPF}(\cdot)$ of controllable DERs and the weighted square errors $C_t^{\textrm{SE}}(\cdot)$ for the state estimation. Assumptions \ref{ass:slater_deterministic}--\ref{ass:C_OPF} straightforwardly hold for problem $(\bm{\mathcal{P}_t^3})$. By including the additional equality constraint $\mathbf{v}_t(\mathbf{u}_t) = \mathbf{v}(\mathbf{z}_t)$, the optimal set-points $\mathbf{u}_t$ of DERs are based on the estimation $\mathbf{z}_t$ to satisfy the targeted voltage regulation. 
% We include The constraint $\mathbf{v}_t(\mathbf{u}_t) = \mathbf{v}_t(\mathbf{z}_t)$ directly bridge (feedback) the optimal estimation results from $(\bm{\mathcal{P}_t^2})$ to the optimal control solution from OPF $(\bm{\mathcal{P}_t^1})$ for voltage regulation.
% {\bf (Changhong: Questions below for this sentence.)
% \begin{enumerate}
% \item First, one should either write $C^{\textrm{SE}}_t(\mathbf{z}_t)$ or more precisely $C^{\textrm{SE}}_t(\mathbf{z}_t, \mathbf{v_t}(\mathbf{z}_t))$ and remove \eqref{eq:p3_se} and remove $\mathbf{v}_t$ from optimization variables, because $\mathbf{v}_t$ is already substituted by $\mathbf{v}_t(\mathbf{z}_t)$; or keep \eqref{eq:p3_se} and $\mathbf{v}_t$ as optimization variables and write the objective function as $C^{\textrm{SE}}_t(\mathbf{z}_t, \mathbf{v}_t)$; 
% \item There is a complete separation between (i) The estimation $\mathbf{v}_t=\mathbf{v}_t(\mathbf{z}_t)$ and (ii) $\mathbf{v}_t(\mathbf{u}_t)$ on which voltage regulation is enforced. Where are they connected?
% Problem $(\mathcal{P}^3_t)$ seems just a trivial combination of $(\mathcal{P}^1_t)$ and $(\mathcal{P}^2_t)$ without any coupling between the two subproblems, i.e., $(\mathcal{P}^3_t)$ can be exactly solved by solving $(\mathcal{P}^1_t)$ and $(\mathcal{P}^2_t)$ separately. Then what is the point of formulating $(\mathcal{P}^3_t)$?  
% \end{enumerate}
% The analysis result below $(\bm{\mathcal{P}_t^3})$ can be solved by  \eqref{eq:algorithm_joint_t} as a primal-dual algorithm.
%
\begin{theorem}\label{thm:reverse_engineering}
% \color{red}
At any given $t>0$, the dynamics \eqref{eq:algorithm_joint_t} serve as a primal-dual gradient algorithm to solve the saddle-point problem of \eqref{eq:opf_se} with Tikhonov regularization terms\footnote{The Tikhonov regularization terms facilitate the convergence performance and ensure that the saddle-point problem is strongly convex on primal variables and strongly concave on dual variables. Note that the equilibrium of the dynamics \eqref{eq:algorithm_joint_t} only approximately solves the original optimization problem \eqref{eq:opf_se} due to the additional regularization terms. The deviation of the approximate solution from the optimal solution is quantified in \cite{koshal2009distributed}.} on the dual variables in the Lagrangian.
\end{theorem}
\begin{proof}
See Appendix \ref{appendix_reverse_engineering}.
\end{proof}
\noindent Theorem \ref{thm:reverse_engineering} demonstrates the optimality and convergence of the proposed algorithm, and provides a way in how to jointly engineer the optimal dispatch and estimation layers as a single optimization problem, which were conventionally separate.
%Now the control/estimation scheme can be designed in a principled way by engineering the optimization problem in multiple directions.
%In the next section, we will explicitly incorporate the noise from the voltage estimator and the linearization error of AC power flow to engineer an online control-estimation scheme that solves a stochastic optimization problem.

\section{Stochastic Optimization-Estimation and Its Online Implementation}\label{sec:stochastic_online}
In this section, we modify the previous OPF-SE synthesis $(\bm{\mathcal{P}^3_t})$ to enhance it to a stochastic problem by incorporating the noise from the voltage estimation and the linearization errors of the AC power flow. We then develop a computationally efficient online solution to the stochastic problem and analyze its convergence and optimality.

\subsection{Stochastic Optimization-Estimation}
\label{sec:fwd}
The modelling and estimation errors can degrade the performance of feedback-based OPF solvers, sometimes even produce infeasible solutions \cite{guo2020solving,guo2020optimal}. 
To address this issue, we extend $(\bm{\mathcal{P}^3_t})$ to a stochastic problem to enable robust decision making.
%to include the stochastic modelling of these noises for robust decisions of DERs. 
In the stochastic formulation, we express the voltage profile as a combination of the linearized voltage model and a random vector $\bm{\xi}_t \in\mathbb{R}^N$ to model the inherent errors of the state estimation and the power flow linearization:
\begin{equation}\nonumber
    \mathbf{\tilde{v}}_t^{\textrm{real}}  = \mathbf{v}_t(\mathbf{u}_t) + \bm{\xi}_t.
\end{equation}
%Let $\mathbf{\tilde{v}}^{\textrm{real}}_t \in \mathbb{R}^N$ denote the real voltage magnitude. 
where the probability distribution of $\bm{\xi}_t$ is unknown but it has finite mean and covariance values for $t>0$. Accordingly, we change the voltage constraint \eqref{eq:p3_voltage_constraint_upper} to a stochastic version by replacing $\mathbf{v}_t(\mathbf{u}_t)$ with $\mathbf{v}_t(\mathbf{u}_t) + \bm{\xi}_t$:
\begin{equation}\label{eq:stochastic_voltage_constraint}\nonumber
\mathbf{r}(\mathbf{v}_t(\mathbf{u}_t) + \bm{\xi}_t) \leq  0.
\end{equation}
%This constraint incorporates the feedback noises with the OPF outputs $\mathbf{u}_t$ to improve the robust performance of OPF controllers. 
We can now formulate a stochastic OPF-SE problem that restricts the risk of voltage violation:
\begin{subequations}\label{eq:opf_se_prob}
\begin{align}
(\bm{\mathcal{P}_t^4}) & \underset{\mathbf{u}_t,\mathbf{z}_t,\mathbf{v}_t}{\min} && \mathbf{E}^{\mathbf{P}}_{\bm{\xi}} \quad \bigg[ C^{\textrm{OPF}}_t(\mathbf{u}_t) + C^{\textrm{SE}}_t(\mathbf{z}_t,\mathbf{v}_t(\mathbf{z}_t))\bigg],\\
& \textrm{subject to:} &&\textrm{Pr}\Big\{\mathbf{r}(\mathbf{v}_t(\mathbf{u}_t) + \bm{\xi}_t)\leq   0 \Big\} \geq 1-\beta, \label{eq:voltreg_random_upper}\\
& && \mathbf{v}_t(\mathbf{u}_t) = \mathbf{v}_t(\mathbf{z}_t),\\
& && \mathbf{u}_t\in\mathcal{X}_t,
\end{align}
\end{subequations}
where operator $\textrm{Pr}\{\cdot\}$ indicates a transformation of the inequality constraint into a chance constraint. The set-points $\mathbf{u}_t$ can be scheduled in a way that voltage limits are satisfied with the prescribed probability $1-\beta$. 
% \color{red}
\begin{remark}[Uncertainties Realization]
There are a variety of classic ways
% \cite{vrakopoulou2013probabilistic,bertsimas2018robust,esfahani2018data,wiesemann2014distributionally}
to reformulate the stochastic OPF-SE problem $(\bm{\mathcal{P}_t^4})$ to obtain tractable subproblems that can be solved by standard convex optimization solvers. These include assuming a specific functional form for the distribution (e.g., Gaussian) of $\bm{\xi}_t$ based on the statistical information of the historical data and using constraint violation risk metrics, such as those encoding value at risk (i.e., chance constraints), conditional value at risk (CVaR), distributional robustness, and support robustness. In addition, the linearization and estimation errors can also be sampled by evaluating the residual between the measurement and estimation results of the quantity of interests (e.g., voltage magnitude). This provides an empirical distribution of $\bm{\xi}_t$ without a prescribed distribution assumption. Hence, the result is a data-based reformulation of the stochastic OPF-SE problem. Since we assume the time-varying stochastic OPF and SE problems are well-posed due to the Slater's condition and full observability, it is obvious that $\bm{\xi}$ is bounded for any $t>0$. 
\end{remark}
% In this paper, the feedback errors $\bm{\xi}_t$ are sampled from a prescribed Gaussian distribution, where the statistical information is attained from the historical data.
% To account for the variety of possible distributions of feedback errors $\bm{\xi}_t$ and yet derive a computationally efficient method for the stochastic OPF-SE synthesis in \eqref{eq:opf_se_prob}, the chance constraint \eqref{eq:voltreg_random_upper} can be reformulated using stochastic optimization technologies, such as scenario-based approaches or moment-based/metric-based distributionally robust optimization. 
The chance constraints for voltage regulation are approximated by leveraging the sample average of CVaR values \cite{rockafellar2000optimization,nemirovski2007convex}. Consider a random variable $z\in\mathbb{R}$ and a scale factor $\tau \in \mathbb{R}_{++}$, we have $\textrm{Pr}(\tau^{-1} z \geq 0) = \textrm{Pr}(z \geq 0) = \mathbb{E}\left[\mathbf{1}_{[0,\infty)}(\tau^{-1} z)\right]$ where $\mathbf{1}_{[0,\infty)}$ is the indication function over the set $[0,\infty)$. The key step to develop a convex approximation for  $\textrm{Pr}(\tau^{-1} z \geq 0) = \textrm{Pr}(z \geq 0)$ is to find a non-negative non-decreasing convex generating function $\psi(\cdot):\mathbb{R} \to \mathbb{R}$ with $\psi(z) > \psi(0) = 1$ for all $z>0$, which is called generating function that will generate a family of convex approximation for the chance constraints due to $\psi(\tau^{-1} z) \geq \mathbf{1}(\tau^{-1} z)$. Having the generating function leads to $\mathbf{E}\psi(\tau^{-1} z) \geq \mathbf{E}[\mathbf{1}_{[0,\infty)}(\tau^{-1} z)] = \textrm{Pr}(z\geq0)$, which is an upper bound on the chance constraint. Now replacing $z$ with $\mathbf{r}(\mathbf{v}(\mathbf{u}) + \bm{\xi})$ without the time index yields $\mathbf{E}[\psi(\tau^{-1}\mathbf{r}(\mathbf{v}(\mathbf{u}) + \bm{\xi}))] \geq \textrm{Pr}(\mathbf{r}(\mathbf{v}(\mathbf{u}) + \bm{\xi}) > 0)$ for all possible $\mathbf{u}$ and $\tau > 0$. We observe that if there exists $\tau > 0$ such that $\tau\mathbf{E}[\psi(\tau^{-1}\mathbf{r}(\mathbf{v}(\mathbf{u}) + \bm{\xi}))]\leq \tau \beta$, then $\textrm{Pr}(\mathbf{r}(\mathbf{v}(\mathbf{u}) + \bm{\xi}) > 0) \leq \beta $. This leads to a sufficient condition $\inf_{\tau>0} \left[\tau\mathbf{E}[\psi(\tau^{-1}\mathbf{r}(\mathbf{v}(\mathbf{u}) + \bm{\xi}))] - \tau\beta\right] \leq 0$ for $\textrm{Pr}(\mathbf{r}(\mathbf{v}(\mathbf{u}) + \bm{\xi}) > 0) \leq \beta $, which is a conservative approximation of the chance constraint \eqref{eq:voltreg_random_upper}. 

Clearly, the above approximation depends on the choice of the generating function $\psi(\cdot)$, such as Markov: $\psi(z) := [1+z]_+$, Chebyshev: $\psi(z) := ([1+z]_+)^2$, Traditional Chebyshev: $\psi(z) := (1+z)^2$ and Chernoff/Berstein: $\psi(z) := e^z$ \cite{nemirovski2007convex}. From the point of view of a tighter approximation with relative smaller value $\psi(z)$ than others, the Markov generating function approximates the chance constraint in the form of 
\begin{equation}\nonumber
\inf_{\tau>0} \Big[\mathbf{E}[\mathbf{r}(\mathbf{v}(\mathbf{u}) + \bm{\xi}) + \tau]_+ - \tau\beta\Big] \leq 0,
\end{equation}
which is related to the concept of Conditional Value at Risk. Quantifying both the frequency and the expected severity of constraint violations, the CVaR serves a widely adopted risk measure for optimization under uncertainties. Note that the Markov generating function is non-decreasing and $\mathbf{r}(\mathbf{u},\bm{\xi})$ is convex in $\mathbf{u}$ given any $\bm{\xi}$, which implies that the above CVaR-based chance constraint is convex in optimization variables. Employing the linear model \eqref{eq:linear_powerflow} and the voltage constraint $\mathbf{r}(\mathbf{v}_t(\mathbf{u}_t))= [(\mathbf{v}_t(\mathbf{u}_t) - \mathbf{v}^{\textrm{max}})^\top, (\mathbf{v}^{\textrm{min}} - \mathbf{v}_t(\mathbf{u}_t))^\top]^\top \leq 0$, a sample averaging CVaR-based convex approximation of the voltage constraint can be formulated as: 
% Quantifying both the frequency and the expected severity of constraint violations, the CVaR serves as a widely adopted risk measure for optimization under uncertainties. We skip the detailed formulation and refer interested readers to \cite{summers2015stochastic,dall2017chance}. Employing the linear model \eqref{eq:linear_powerflow} and the voltage constraint $\mathbf{r}(\mathbf{v}_t(\mathbf{u}_t))= [(\mathbf{v}_t(\mathbf{u}_t) - \mathbf{v}^{\textrm{max}})^\top, (\mathbf{v}^{\textrm{min}} - \mathbf{v}_t(\mathbf{u}_t))^\top]^\top \leq 0$, a CVaR-based convex approximation of the voltage constraint can be formulated as:
\begin{equation} \nonumber
\begin{aligned}
    &\frac{1}{N_s}\sum_{s=1}^{N_s}\Big[\mathbf{v}_t(\mathbf{u}_t) - \mathbf{v}^{\textrm{max}} + \bm{\xi}_t^s + \overline{\bm{\tau}}_t \Big]_+ - \overline{\bm{\tau}}_t\beta \leq 0, \quad\forall s\in\Xi_t, \\
    &\frac{1}{N_s}\sum_{s=1}^{N_s}\Big[\mathbf{v}^{\textrm{min}} - \mathbf{v}_t(\mathbf{u}_t) - \bm{\xi}_t^s + \underline{\bm{\tau}_t} \Big]_+ - \underline{\bm{\tau}}_t\beta \leq 0, \quad\forall s\in\Xi_t,
\end{aligned}
\end{equation}
where vectors $\overline{\bm{\tau}}_t\in\mathbb{R}^N,\underline{\bm{\tau}}_t\in\mathbb{R}^N$ are the CVaR auxiliary variables \cite{rockafellar2000optimization}. Given $N_s$ samples $\Xi_t:=\{\bm{\xi}_t^s\}_{s=1}^{N_s}$ of random variable $\bm{\xi}_t$ at time $t$, the expectation of the left-hand side of \eqref{eq:voltreg_random_upper} is obtained via sample averaging. The sample average approximation (SAA) methods with a modest number of samples has been shown to be effective in many applications \cite{bertsimas2018robust,dall2017chance}. 
% \begin{assumption}\label{ass:error_bound}
% The feedback noise $\xi\in\mathbb{R}^N$m including estimation and linearization errors are bounded by $\mathbf{e}_{\xi}>0$ for any $t>0$, given by
% \begin{equation}
%     \bm{\xi}_t \leq \mathbf{e}_{\xi}.
% \end{equation}
% \end{assumption}
An approximation of \eqref{eq:voltreg_random_upper} for an arbitrary distribution can be
accommodated in the time-varying optimization-estimation synthesis problem as follows:
\begin{subequations}\label{eq:opf_se_cvar}
\begin{align}
(\bm{\mathcal{P}^5_t}) & \underset{\mathbf{u}_t,\mathbf{z}_t,\mathbf{v}_t,\bm{\tau}_t}{\min} &&  C_t^{\textrm{OPF}}(\mathbf{u}_t) + C_t^{\textrm{SE}}(\mathbf{z}_t,\mathbf{v}_t(\mathbf{z}_t)), \label{eq:opf_se_cvar_objective}\\
& \textrm{subject to:} && \mathbf{g}_{t}(\mathbf{v}_t(\mathbf{u}_t),\bm{\tau}_t) \leq 0, \label{eq:stochastic_opf_se_voltage_equality}\\  
% &  &&\overline{g}_{t,j}(\mathbf{u}_t,\overline{\bm{\tau}}_t) \leq 0, \label{eq:stochastic_opf_se_voltage_inequality1}\\
&  && \mathbf{v}_t(\mathbf{u}_t) = \mathbf{v}_t(\mathbf{z}_t), \label{eq:stochastic_opf_se_voltage_inequality2}\\
& && \mathbf{u}_t\in\mathcal{X}_t,
\end{align}
\end{subequations}
where $\mathbf{g}_{t}(\mathbf{v}_t(\mathbf{u}_t),\bm{\tau}_t):=[\overline{\mathbf{g}}_t(\mathbf{v}_t(\mathbf{u}_t),\overline{\bm{\tau}}_t)^\top,\underline{\mathbf{g}}_t(\mathbf{v}_t(\mathbf{u}_t),\underline{\bm{\tau}}_t)^\top]^\top \in \mathbb{R}^{2N}$, $\bm{\tau}_t:=[\underline{\bm{\tau}}_t^\top,\overline{\bm{\tau}}_t^\top]^\top \in \mathbb{R}_+^{2N}$ and
\begin{subequations} \label{eq:CVaR_v}\nonumber
    \begin{align}
        & \overline{\mathbf{g}}_t(\mathbf{v}_t(\mathbf{u}_t),\overline{\bm{\tau}}_t)  = \\
        & ~~~~~~~~~~~~ \frac{1}{N_s}\sum_{s=1}^{N_s}\Big[\mathbf{v}_t(\mathbf{u}_t) - \mathbf{v}^{\textrm{max}} + \bm{\xi}_t^s + \overline{\bm{\tau}}_t \Big]_+ - \overline{\bm{\tau}}_t\beta,\\
        & \underline{\mathbf{g}}_t(\mathbf{v}_t(\mathbf{u}_t),\underline{\bm{\tau}}_t) = \\ 
        & ~~~~~~~~~~~~ \frac{1}{N_s}\sum_{s=1}^{N_s}\Big[\mathbf{v}^{\textrm{min}} - \mathbf{v}_t(\mathbf{u}_t) - \bm{\xi}_t^s + \underline{\bm{\tau}}_t \Big]_+ - \underline{\bm{\tau}}_t\beta.
    \end{align}
\end{subequations}
\begin{assumption}[Slater's Condition]\label{ass:slater_stochastic}
There exist a set of feasible inputs $\mathbf{u}_t \in\mathcal{X}_t$ and a set of auxiliary variables $\bm{\tau}_t \in\mathbb{R}_+$, such that $\mathbf{g}_t(\mathbf{u}_t,\bm{\tau}_t) \leq 0$ for all $t>0$. 
\end{assumption}
\noindent Similar to Assumption \ref{ass:slater_deterministic}, the Slater's condition above ensures strong duality and the feasibility of the well-posed stochastic joint OPF-SE problem \eqref{eq:opf_se_cvar} for a well-designed distribution network with high penetration of renewables.

\subsection{Regularization Errors Analysis}
% \begin{assumption}[Slater's Condition]
% \label{ass:slater_stochastic} For any $t>0$, there exists a strictly feasible point within the operation region $(\mathbf{u}_t,\bm{\tau}_t) \in \mathcal{X}_t \times \mathbb{R}_+$, such that
% \begin{equation}
%   \mathbf{g}_t(\mathbf{u}_t,\bm{\tau}_t) \leq 0.
% \end{equation}
% \end{assumption}

Let $\mathcal{L}_t(\mathbf{u}_t,\mathbf{z}_t,\bm{\tau}_t,\bm{\lambda}_t)$ be the Lagrangian of \eqref{eq:opf_se_cvar}, where $\bm{\lambda}_t\in\mathbb{R}_+^{2N}$ denotes the Lagrange multiplier associated with voltage constraint \eqref{eq:stochastic_opf_se_voltage_inequality2}, i.e.,
% \begin{equation} \label{eq:L_opf_se_timevarying}
%     \begin{aligned}
%     \mathcal{L}(\mathbf{u}^t,\mathbf{z}^t,\bm{\tau}^t,\bm{\gamma}^t) := f^t(\mathbf{u}^t) + g^t(\mathbf{z}^t) + (\overline{\bm{\gamma}}^t)^\top\overline{\bm{\kappa}}^t(\mathbf{u}^t, \overline{\bm{\tau}}^t) \\ 
%     + (\underline{\bm{\gamma}}^t)^\top\underline{\bm{\kappa}}^t(\mathbf{u}^t, \underline{\bm{\tau}}^t) - \frac{\eta}{2}\left(\|\overline{\bm{\gamma}}^t\|_2^2+\|\underline{\bm{\gamma}}^t\|_2^2\right).
%     \end{aligned}
% \end{equation}
\begin{equation}\label{eq:L_stochastic_opf_se_no-regularizer}
\begin{aligned}
    & \mathcal{L}_t(\mathbf{u}_t,\mathbf{z}_t,\bm{\tau}_t,\bm{\lambda}_t) = \\
    & ~~~~~~~~~~~~C^{\textrm{OPF}}_t(\mathbf{u}_t) + C^{\textrm{SE}}_t(\mathbf{z}_t,\mathbf{v}_t(\mathbf{z}_t)) + \bm{\lambda}_t^\top \mathbf{g}_t(\mathbf{u}_t,\bm{\tau}_t).
    \end{aligned}
\end{equation}
The optimal dual variables lie in a compact set due to the compact set $\mathcal{X}_t$ and the Slater's condition in Assumption \ref{ass:slater_stochastic} for all $t>0$.
The Lagrangian \eqref{eq:L_stochastic_opf_se_no-regularizer} is not strongly convex in the CVaR auxiliary variable $\bm{\tau}_t$ and not strongly concave in the dual variable $\bm{\lambda}_t$. To attain a better convergence performance of the gradient approach under time-varying settings, we consider the regularized Lagrangian:
\begin{equation}\label{eq:stochastic_L_opf_se_phi_v}
\begin{aligned}
    &\mathcal{L}_{\phi,v,t}(\mathbf{u}_t,\mathbf{z}_t,\bm{\tau}_t,\bm{\lambda}_t) = \\
    &~~~~~~~~~~~~~~~~~\mathcal{L}_{t}(\mathbf{u}_t,\mathbf{z}_t,\bm{\tau}_t,\bm{\lambda}_t) + \frac{v}{2}\|\bm{\tau}_t\|^2_2 - \frac{\phi}{2}\|\bm{\lambda}_t\|^2_2,
    \end{aligned}
 \end{equation}
with small constants $v>0$ and $\phi>0$ in the Tikhonov regularization terms. 
%for the CVaR auxiliary variable $\bm{\tau}_t$ and the dual variable $\bm{\lambda}_t$, respectively. 
The regularized Lagrangian \eqref{eq:stochastic_L_opf_se_phi_v} is now strictly convex in all primal variables $(\mathbf{u}_t, \mathbf{z}_t, \bm{\tau}_t)$ and strictly concave in the dual variable $\bm{\lambda}_t$. 
The key advantage of the regularized Lagrangian is that it can leverage a gradient approach to attain an approximate solution of $(\bm{\mathcal{P}}_t^5)$ with improved convergence properties within an online implementation. To simplify notation, we define $\eta: = (\phi,v)$. 
%When the interplay between the parameters is relevant, we will write them explicitly. 
We consider the following saddle-point problem:
\begin{equation}\label{eq:saddle_point_problem_regularized_phi_v}
    \max_{\bm{\lambda}_t\in\mathbb{R}^N_{+}}\min_{\begin{subarray}{c}\mathbf{u}_t\in\mathcal{X}_t,\\\mathbf{z}_t\in\mathcal{X}_t,\bm{\tau}_t\in\mathbb{R}^{2N}_+\end{subarray}} \mathcal{L}_{\eta,t} (\mathbf{u}_t,\mathbf{z}_t,\bm{\tau}_t,\bm{\lambda}_t),
\end{equation}
and let $(\mathbf{u}_{\eta,t}^*,\mathbf{z}_{\eta,t}^*,\bm{\tau}_{\eta,t}^*,\bm{\lambda}_{\eta,t}^*)$ denote the unique saddle-point of \eqref{eq:saddle_point_problem_regularized_phi_v} at time $t$. In general, the solutions to \eqref{eq:opf_se_cvar} and \eqref{eq:saddle_point_problem_regularized_phi_v} are different due to the regularization terms. This difference can be bounded as shown later. 
We first state the following assumption.
%state the primary assumptions on the objective function and constraints in the original problem \eqref{eq:opf_se_cvar}. 
% \color{red}
\begin{assumption}[Lipschitz Property] 
\label{ass:Lipschitz_Property_gradient_bound}
Define $\mathbf{x}_t:=[\mathbf{u}^\top_t,\mathbf{z}^\top_t,\bm{\tau}^\top_t]^\top$ 
and $\tilde{f}_t(\mathbf{x}_t):=C_t^{\textrm{OPF}}(\mathbf{x}_t) + C_t^{\textrm{SE}}(\mathbf{x}_t) + \frac{v}{2}\|\bm{\tau}_t\|^2_2$. There exist positive constants $G_f$ and $G_g$, such that the gradient of the convex and differentiable function $\nabla_{\mathbf{x}}\tilde{f}_t(\mathbf{x}_t)$ and any subgradient of the piece-wise linear constraints $\nabla_{\mathbf{x}}\mathbf{g}_t(\mathbf{x}_t)\in\partial_{\mathbf{x}} \mathbf{g}_t(\mathbf{x}_t)$ are bounded by
%and $\mathbf{g}_t(\mathbf{x}_t)$ are strongly convex.
%The gradients/subgradients $\nabla_{\mathbf{x}} \tilde{f}_t(\mathbf{x}_t) \in \partial_{\mathbf{x}} \tilde{f}_t(\mathbf{x}_t)$ and $\nabla_{\mathbf{x}} \mathbf{g}_t(\mathbf{x}_t) \in \partial_{\mathbf{x}} \mathbf{g}_t(\mathbf{x}_t)$ are bounded by
\begin{equation}\nonumber
    \|\nabla_{\mathbf{x}} \tilde{f}_t(\mathbf{x}_t)\|_2 \leq G_f, \quad \|\nabla_{\mathbf{x}} \mathbf{g}_t(\mathbf{x}_t)\|_2 \leq G_g,
\end{equation}
for all $\mathbf{x}_t\in\mathcal{X}_t \times \mathbb{R}_+^{2N}$ and $t>0$.
\end{assumption}
The constant $G_f$ exists if Assumption \ref{ass:C_OPF} holds and the piece-wise function $\mathbf{g}_t(\cdot)$ in $\mathbf{x}_t$ implies the existence of the constant $G_g$.
Let $\mathbf{x}^*_t:= [(\mathbf{u}^*_t)^\top,(\mathbf{z}^*_t)^\top,(\bm{\tau}^*_t)^\top]^\top$ be an optimal solution to problem \eqref{eq:opf_se_cvar} and $\mathbf{x}_{\eta,t}^* := [(\mathbf{u}_{\eta,t}^*)^\top,(\mathbf{z}_{\eta,t}^*)^\top,(\bm{\tau}_{\eta,t}^*)^\top]^\top$ be the unique primal optimal of the regularized problem \eqref{eq:saddle_point_problem_regularized_phi_v}.
%We now provide an upper bound on the distance between $\mathbf{x}_t^*$ and $\mathbf{x}_{\eta,t}^*$. 

\begin{theorem}[Regularization Error]
\label{thm:performance_bound}
For all $t>0$, the difference between $\mathbf{x}_{t}^*$ and $\mathbf{x}_{\eta,t}^*$ is bounded by
\begin{equation}\label{eq:regularizer_bound}
\begin{aligned}
&\|\mathbf{x}^*_t - \mathbf{x}_{\eta,t}^*\|^2_2 \leq \\ 
&~~~~~\left(\frac{2\left(G_f + G_g \|\bm{\lambda}_{v,t}^*\|_1\right)}{c}\right)^2 + \frac{\phi}{2c}\left(\|\bm{\lambda}^*_{v,t}\|_2^2 - \|\bm{\lambda}_{\eta,t}^*\|^2_2\right).
\end{aligned}
\end{equation}
where $\bm{\lambda}_{v,t}^*$ and $\bm{\lambda}_{\eta,t}^*$ are dual optimal solutions of \eqref{eq:saddle_point_problem_regularized_phi_v} for $\phi = 0$ \& $v > 0$ and $\phi > 0$ \& $v > 0$, respectively.
\end{theorem}
\begin{proof}
See Appendix \ref{appendix_performance_bound}.
\end{proof}
\noindent By Slater's condition, the dual optimal sets $\Lambda_{v,t}$ and $\Lambda_{\eta,t}$ of \eqref{eq:saddle_point_problem_regularized_phi_v} are nonempty for $\phi = 0$ \& $v > 0$ and $\phi > 0$ \& $v > 0$, respectively. In what follows, let $\bm{\lambda}_{v,t}^* \in \Lambda_{v,t}$ and $\bm{\lambda}_{\eta,t}^* \in \Lambda_{\eta,t}$ be arbitrary but fixed dual optimal solutions for \eqref{eq:saddle_point_problem_regularized_phi_v}. The bound in \eqref{eq:regularizer_bound} is jointly dependent on the Lipschitz properties $(G_f,G_g)$ of the objective $\tilde{f}(\mathbf{x}):=f(\tilde{\mathbf{x}}) + \frac{v}{2}\|\bm{\tau}\|^2_2$ and the constraint $\mathbf{g}_t(\mathbf{x}_t)$, as well as the growth property of the objective function $\tilde{f}(\mathbf{x})$ defined by the constant $c>0$.

\subsection{Online Joint OPF-SE Algorithm}
In this subsection, we present an online solution to the time-varying stochastic OPF-SE synthesis problem \eqref{eq:opf_se_cvar}, which enables the DERs to respond timely to variations in system states. In particular, we propose an online primal-dual gradient-based algorithm to solve the regularized stochastic OPF-SE problem \eqref{eq:saddle_point_problem_regularized_phi_v}. The DER dispatch decisions and voltage estimators are iteratively updated in real time towards an approximate solution to the time-varying problem \eqref{eq:opf_se_cvar}. The gradient updates are presented in detail as follows:
% To seek practical scenario, the online outputs of DERs incorporate the feedback loops noises for better robust and feasibility performance.
% Traditionally speaking, the control and estimation layers are decoupled in the physical systems. the online synthesis architecture incorporate online estimation feedback to 
%Consider the following primal-dual gradient updates to solve the time-varying regularized saddle-point problem \eqref{eq:saddle_point_problem_regularized_phi_v}
\begin{subequations}\label{eq:stochastic_opf_se_controller}
\begin{align}
\mathbf{u}_{\eta,t+1} & =\left[\mathbf{u}_{\eta,t} - \epsilon \nabla_{\mathbf{u}}\mathcal{L}_{\eta,t}(\mathbf{u},\mathbf{z},\bm{\tau},\bm{\lambda})\Big|_{\begin{subarray}\mathbf{u}_{\eta,t},\mathbf{z}_{\eta, t},\\\bm{\tau}_{\eta, t},\bm{\lambda}_{\eta,t}\end{subarray}} \right]_{\mathcal{X}_t},\label{eq:stochastic_opf_se_controller_opf}\\  
\bm{\tau}_{\eta, t+1} & = \left[\bm{\tau}_{\eta,t} - \epsilon \nabla_{\bm{\tau}}\mathcal{L}_{\eta,t}(\mathbf{u},\mathbf{z},\bm{\tau},\bm{\lambda})\Big|_{\begin{subarray}\mathbf{u}_{\eta,t},\mathbf{z}_{\eta,t},\\\bm{\tau}_{\eta,t},\bm{\lambda}_{\eta,t}\end{subarray}}\right]_{\mathbb{R}_+}, \label{eq:stochastic_opf_se_controller_cvar}\\
\bm{\lambda}_{\eta,t+1} & = \bigg[\bm{\lambda}_{\eta,t} + \epsilon \left(\mathbf{g}_t(\mathbf{v}_t(\mathbf{z}_{\eta,t}), \bm{\tau}_{\eta,t}) - \phi\bm{\lambda}_{\eta,t}\right) \bigg]_{\mathbb{R}_+},\label{eq:stochastic_opf_se_controller_lambda}\\
\mathbf{z}_{\eta,t+1} & = \mathbf{z}_{\eta,t} - \epsilon \nabla_{\mathbf{z}}\mathcal{L}_{\eta,t}(\mathbf{u},\mathbf{z},\bm{\tau},\bm{\lambda})\Big|_{\begin{subarray}\mathbf{u}_{\eta,t},\mathbf{z}_{\eta,t},\\\bm{\tau}_{\eta,t},\bm{\lambda}_{\eta,t}\end{subarray}}, \label{eq:stochastic_opf_se_controller_se}
\end{align}
\end{subequations}
where $\epsilon >0$ is a constant step size. 
%We run the following optimal control-estimation synthesis shown in Algorithm \ref{algorithm:online_OPF-SE algorithm} over time for $t>0$, to attain the online set-points $\mathbf{u}_{t}$ of DERs and the online estimated voltage profile $\mathbf{v}_{t}(\mathbf{z}_{t})$.
\begin{algorithm}
    \caption{(Online Joint OPF-SE Algorithm)}\label{algorithm:online_OPF-SE algorithm}
    \begin{algorithmic}[1]
        \Require[S0] Initialize the DER dispatch set-points and estimation variables $\{\mathbf{u}_{\eta,0},\mathbf{z}_{\eta,0}\}$, CVaR auxiliary variables $\bm{\tau}_{\eta,0}$ and dual variables $\bm{\lambda}_{\eta,0}$.
        \While{$t = 0:T$}
            \State[S1] Network operator collects the voltage magnitude measurements $\hat{v}_{i,t}, \forall i\in\mathcal{M}_v$, the load pseudo-measurements $\hat{p}_{i,t},\forall i\in\mathcal{M}_p$ and $\hat{q}_{i,t},\forall i\in\mathcal{M}_q$.
            \State[S2] Network operator updates the estimation variables \eqref{eq:stochastic_opf_se_controller_se}.
            \State[S3] Network operator utilizes the estimation feedback $\mathbf{v}_t =  \mathbf{v}_t(\mathbf{z}_{\eta,t})$ to update the dual variables \eqref{eq:stochastic_opf_se_controller_lambda}.
            \State[S4] Network operator transmits the dual variables $\bm{\lambda}_t$ to end-users.
            \State[S5] End-users update the dispatch decisions \eqref{eq:stochastic_opf_se_controller_opf} and the CVaR auxiliary variables \eqref{eq:stochastic_opf_se_controller_cvar}.
        \EndWhile
\end{algorithmic}
\end{algorithm}
\noindent For the time-varying case at hand, the online optimization-estimation Algorithm \ref{algorithm:online_OPF-SE algorithm} is able to capture the voltage trajectories due to the
variability of underlying network states (i.e., renewables generation and loads). Before analyzing the convergence of Algorithm \ref{algorithm:online_OPF-SE algorithm}, we introduce additional assumptions to limit the difference between the optimization solutions of consecutive time instants.
\begin{assumption}\label{ass:bound_primal}
For all $t>0$, there exist constants $\sigma_{\mathbf{u}} \geq 0$, $\sigma_{\mathbf{z}} \geq 0$ and $\sigma_{\bm{\tau}} \geq 0$ to bound the primal optimizer of \eqref{eq:saddle_point_problem_regularized_phi_v} at two consecutive time steps, i.e., $\|\mathbf{u}^*_{\eta,t+1} - \mathbf{u}^*_{\eta,t}\|_2 \leq \sigma_{\mathbf{u}}$, $\|\mathbf{z}^*_{\eta,t+1} - \mathbf{z}^*_{\eta,t}\|_2 \leq \sigma_{\mathbf{z}}$ and $\|\bm{\tau}^*_{\eta,t+1} - \bm{\tau}^*_{\eta,t}\|_2 \leq \sigma_{\bm{\tau}}$.
\end{assumption}
\begin{assumption}\label{ass:bound_constraints}
The difference between constraint \eqref{eq:stochastic_opf_se_voltage_equality} at the optimal solutions of two consecutive time steps is bounded by constant $\sigma_{\mathbf{g}}\geq 0$ for all $t>0$, i.e.,
\begin{equation*}
\Big\|\mathbf{g}_{t+1}(\mathbf{v}_{t+1}(\mathbf{u}_{\eta,t+1}^*), \bm{\tau}_{\eta,t+1}^*) -\mathbf{g}_{t}(\mathbf{v}_t(\mathbf{u}_{\eta,t}^*), \bm{\tau}_{\eta,t}^*) \Big\|_1 \leq \sigma_{\mathbf{g}},\\
\end{equation*}
\begin{equation*}
\Big\|\mathbf{g}_{t+1}(\mathbf{v}_{t+1}(\mathbf{z}_{\eta,t+1}^*), \bm{\tau}_{\eta,t+1}^*) -\mathbf{g}_{t}(\mathbf{v}_t(\mathbf{z}_{\eta,t}^*), \bm{\tau}_{\eta,t}^*) \Big\|_1 \leq \sigma_{\mathbf{g}}.\\
\end{equation*}
\end{assumption}
Assumptions \ref{ass:bound_primal} and \ref{ass:bound_constraints} define the variability of the primal variables and constraint space of the time-varying saddle-point problem \eqref{eq:saddle_point_problem_regularized_phi_v}. Such variability can also be defined with respect to dual variables, i.e.,  $\|\bm{\lambda}_{\eta,t+1}^* - \bm{\lambda}_{\eta,t}^*\| \leq \sigma_{\bm{\lambda}}$
%, which was investigated by the optimality condition of \eqref{eq:saddle_point_problem_regularized_phi_v} 
\cite{simonetto2014double}.
Define $\mathbf{e}_{\eta,t}^*:=[(\mathbf{u}_{\eta,t}^*)^\top, (\mathbf{z}_{\eta,t}^*)^\top, (\bm{\tau}_{\eta,t}^*)^\top, (\bm{\lambda}_{\eta,t}^*)^\top]^\top$, and $\|\mathbf{e}_{\eta,t}^* - \mathbf{e}_{\eta,t-1}^*\| \leq \sigma_\mathbf{e}$, for some $\sigma_{\mathbf{e}} \geq 0$. The variability bounds in Assumptions \ref{ass:bound_primal} and \ref{ass:bound_constraints} are closely related to the time-varying PV and load inputs, which always exist due to the DERs and loads physical capacity. For an online implementation with a relatively small time interval (in seconds) between two consecutive control commands, such variability is small due to gradually ramping PV and load inputs. For convenience in further analysis, we define a time-varying operator: 
%before further investigate the online convergence results for the online control-estimation algorithm
\begin{equation}\label{eq:stochastic_online_gradient_operator}
\begin{aligned}
    &\bm{\Pi}_{\eta,t}: \{\mathbf{u}_\eta,\bm{\tau}_\eta,\bm{\lambda}_\eta\,\mathbf{z}_\eta\}:\longmapsto \\
    &~~~~~~~~~~~~~~~~~\begin{bmatrix}
    \nabla_{\mathbf{u}}\mathcal{L}_{\eta,t}(\mathbf{u},\mathbf{z},\bm{\tau},\bm{\lambda})|_{\mathbf{u}_{\eta,t},\mathbf{z}_{\eta,t},\bm{\tau}_{\eta,t},\bm{\lambda}_{\eta,t}}\\
    \nabla_{\bm{\tau}}\mathcal{L}_{\eta,t}(\mathbf{u},\mathbf{z},\bm{\tau},\bm{\lambda})|_{\mathbf{u}_{\eta,t},\mathbf{z}_{\eta,t},\bm{\tau}_{\eta,t},\bm{\lambda}_{\eta,t}}\\
    -\left(\mathbf{g}_t(\mathbf{v}_t(\mathbf{z}_{\eta,t}), \bm{\tau}_{\eta,t})  - \phi\bm{\lambda}_{\eta,t}\right)\\
    \nabla_{\mathbf{z}}\mathcal{L}_{\eta,t}(\mathbf{u},\mathbf{z},\bm{\tau},\bm{\lambda})|_{\mathbf{u}_{\eta,t},\mathbf{z}_{\eta,t},\bm{\tau}_{\eta,t},\bm{\lambda}_{\eta,t}}
    \end{bmatrix},
\end{aligned}
\end{equation}
which is utilized to compute the gradients in the time-varying updates \eqref{eq:stochastic_opf_se_controller} as
\begin{equation}\nonumber
    \mathbf{e}_{\eta,t+1} = \bigg[\mathbf{e}_{\eta, t} - \epsilon\bm{\Pi}_{\eta,t}\left(\mathbf{e}_{\eta,t}\right) \bigg]_{\mathcal{X}_t\times\mathbb{R}_+\times\mathbb{R}}.
\end{equation}
%then the following properties hold
\begin{lemma}\label{lma:LM_stochastic}
Given Assumptions \ref{ass:C_OPF} and \ref{ass:slater_stochastic}, the operator \eqref{eq:stochastic_online_gradient_operator} is strongly monotone and Lipschitz continuous, i.e., there are constants $\widetilde{M}>0$ and $\widetilde{L}>0$ such that:
\begin{equation}\label{eq:property_Monotone_stochastic}
\left(\bm{\Pi}_{\eta,t}(\mathbf{e}) - \bm{\Pi}_{\eta,t}(\mathbf{e}')\right)^\top \left(\mathbf{e} - \mathbf{e}'\right) \geq \widetilde{M}\|\mathbf{e} - \mathbf{e}'\|_2^2,
\end{equation}
\begin{equation}\label{eq:property_Lipschitz_stochastic}
\|\bm{\Pi}_{\eta,t}(\mathbf{e}) - \bm{\Pi}_{\eta,t}(\mathbf{e}') \|_2^2 \leq \widetilde{L}^2 \|\mathbf{e} - \mathbf{e}'\|_2^2.
\end{equation}
for any feasible $\mathbf{e}$ and $\mathbf{e}'$.
\end{lemma}

\begin{lemma} \label{lma:step_size_stochastic}
%Consider the time-varying primal-dual gradient algorithm \eqref{eq:stochastic_opf_se_controller} for the saddle-point problem \eqref{eq:saddle_point_problem_regularized_phi_v}. 
If the step size $\epsilon$ satisfies
\begin{equation} 
0 < \epsilon < 2\widetilde{M}/\widetilde{L}^2,
\end{equation}
then Algorithm \ref{algorithm:online_OPF-SE algorithm} converges to the unique saddle point of \eqref{eq:saddle_point_problem_regularized_phi_v}.
\end{lemma}

\begin{proof}
Note that the objective function in   \eqref{eq:opf_se_cvar} $C_t^{\textrm{OPF}}(\mathbf{u}_t) + C_t^{\textrm{SE}}(\mathbf{z}_t,\mathbf{v}_t(\mathbf{z}_t))$ is convex quadratic and the associated constraint function $\mathbf{g}_{t}(\mathbf{v}_t(\mathbf{u}_t),\bm{\tau}_t)$ is piece-wise linear, which implies that the saddle-point problem \eqref{eq:saddle_point_problem_regularized_phi_v} with regularizers is strongly convex on all primal variables and strongly concave on all dual variables. This leads to the strongly monotone and Lipschitz continuous properties of the gradient operator \eqref{eq:stochastic_online_gradient_operator}. Similar to Lemma \ref{lem:stepsize}, the step size criteria for the convergence of Algorithm can be easily established by following the analysis procedure in our existing works \cite{zhou2019accelerated,guo2020solving}.
\end{proof}

To conclude, the convergence and online tracking performance of Algorithm \ref{algorithm:online_OPF-SE algorithm} are provided next.

\begin{theorem}[Convergence Analysis]
Given any step size $0 < \epsilon < 2\widetilde{M}/\widetilde{L}^2$, the distance between the sequence $\{\mathbf{e}_{\eta,t}\}$ generated by Algorithm \ref{algorithm:online_OPF-SE algorithm} and the unique saddle-point $\mathbf{e}_{\eta,t}^*$ of the time-varying problem \eqref{eq:saddle_point_problem_regularized_phi_v} is bounded as
\begin{equation}\label{eq:online_bound_thm}
    \lim_{t\to \infty}\sup \|\mathbf{e}_{\eta,t} - \mathbf{e}_{\eta,t}^*\|_2 = \frac{\sigma_\mathbf{e}}{\sqrt{1 - 2\epsilon\widetilde{M} + \epsilon^2\widetilde{L}^2}}.
\end{equation}
\label{thm:online_convergence}
\end{theorem}
\begin{proof}
See Appendix \ref{appendix_online_convergence}.
\end{proof}
%The above condition \eqref{eq:online_bound_thm} quantities the maximum discrepancy between the online control-estimation synthesis results $ \{\mathbf{u}_{\eta,t},\mathbf{z}_{\eta,t},\bm{\tau}_{\eta,t},\bm{\lambda}_{\eta,t}\}$ generated by the proposed algorithm \eqref{eq:stochastic_opf_se_controller} and the time-varying optimizer of the saddle-point problem \eqref{eq:saddle_point_problem_regularized_phi_v}. 
Equation \eqref{eq:online_bound_thm} gives the maximum difference between the online optimization-estimation trajectories generated by Algorithm \ref{algorithm:online_OPF-SE algorithm} and the optimizer of the time-varying stochastic OPF-SE problem \eqref{eq:opf_se_cvar}. This bound depends on the variability of two consecutive time-varying problems (i.e., $\sigma_{\mathbf{e}}$), and a ``small enough" step-size $\epsilon$. As demonstrated in the numerical studies, in the next section, this difference between two consecutive time steps is small if we run the online updates in a fast time-scale (i.e., 1-second). In addition, the ``small enough" step-size $\epsilon$ is trivial to decide by trial-and-error, as typically done in distributed optimization \cite{bertsekas2015parallel}. Overall, Theorems \ref{thm:performance_bound} and \ref{thm:online_convergence} verify the optimality and convergence of the proposed online algorithm with respect to the time-varying stochastic OPF-SE problem \eqref{eq:opf_se_cvar}. These two bounds also illustrate the tradeoffs between the convergence performance and optimality due to the regularization and online implementation.

% \begin{remark}[Asynchronous Updates]
% So far, we have proposed a joint estimation-control framework assuming that 
% mention (1) asynchronous update, and (2) multiple primal updates and one dual update or dual update.
% \end{remark}
% \end{remark}
% \begin{remark}[Hardware In-the-Loop]
% \color{red}
% @Xinyang, we can also comments on our Hardware in-the-loop results to support and distinguish our methodologies from both theoretical and practical implementation perspectives.
% \color{black}
% \end{remark}

\section{Numerical Results}\label{sec:numerical}
A modified IEEE-37 node test feeder is used to demonstrate our proposed online joint optimization-estimation architecture. Fig.~\ref{fig:37node} gives the modified network as a single-phase equivalent with high penetration of distributed photo-voltaic (PV) systems. The line impedance, shunt admittance and locations of active and reactive loads are derived from the dataset \cite{zimmerman2011matpower}. We deploy the proposed online OPF-SE synthesis algorithm, specifically, to mitigate the overvoltage when the PV availability exceeds the consumption. There are 18 PV systems located at nodes 4, 7, 13, 17, 20, 22, 23, 26, 28, 29, 30, 31, 32, 33, 34, 35, and 36, and their availability power is proportional to the real irradiance data with 1-second granularity in \cite{solardata}. The load profiles are replaced by the real measurements with 1-second resolution from feeders in Anatolia, CA during the week of August, 2012 in  \cite{bank2013development}. The ratings of the inverters are all 200 kVA, except for the inverters at nodes 3, 15 and 16, which are 340 kVA.

\begin{figure}[!htbp]
    \centering
    \includegraphics[width=3.5in]{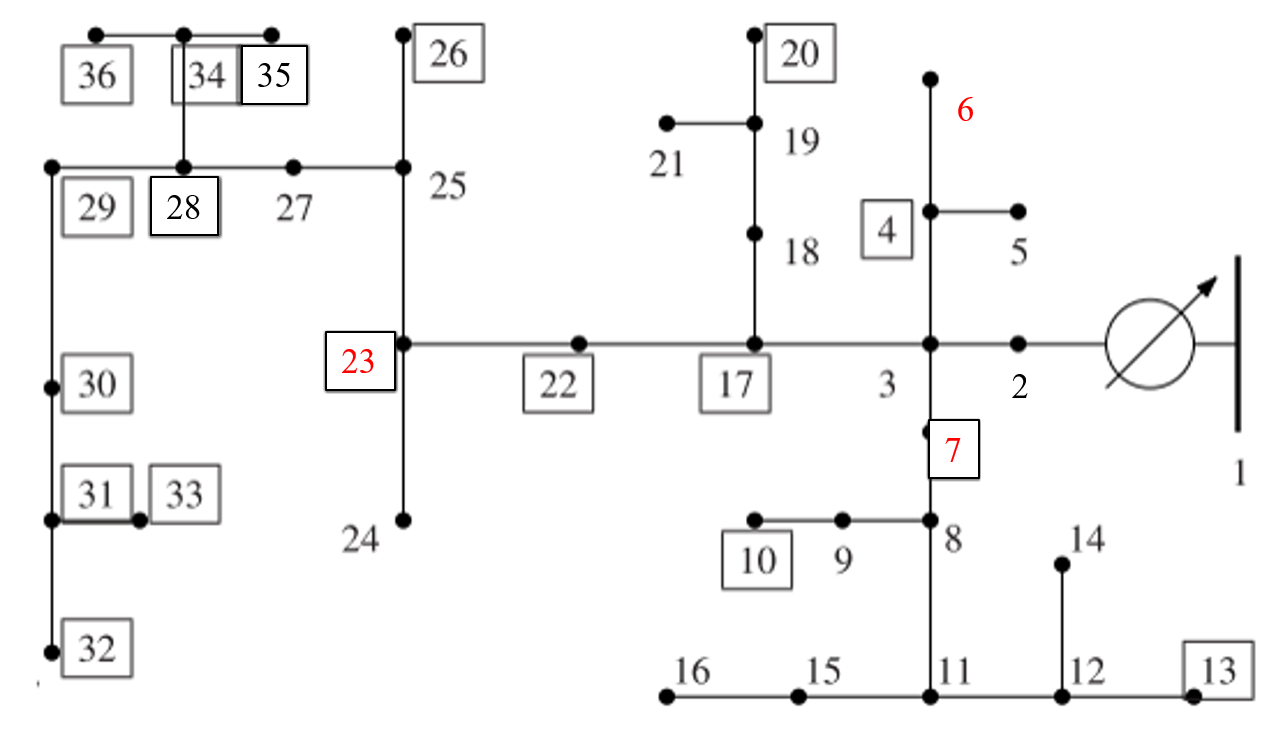}
    \caption{IEEE 37-node distribution network. The boxes indicate the nodes with PV generation. The locations of voltage sensors (i.e., $i = 6, 7$ and $23$) are color-coded in red.}
    \label{fig:37node}
\end{figure}
\begin{figure*}[h]
    \centering
    \includegraphics[width=7.1in]{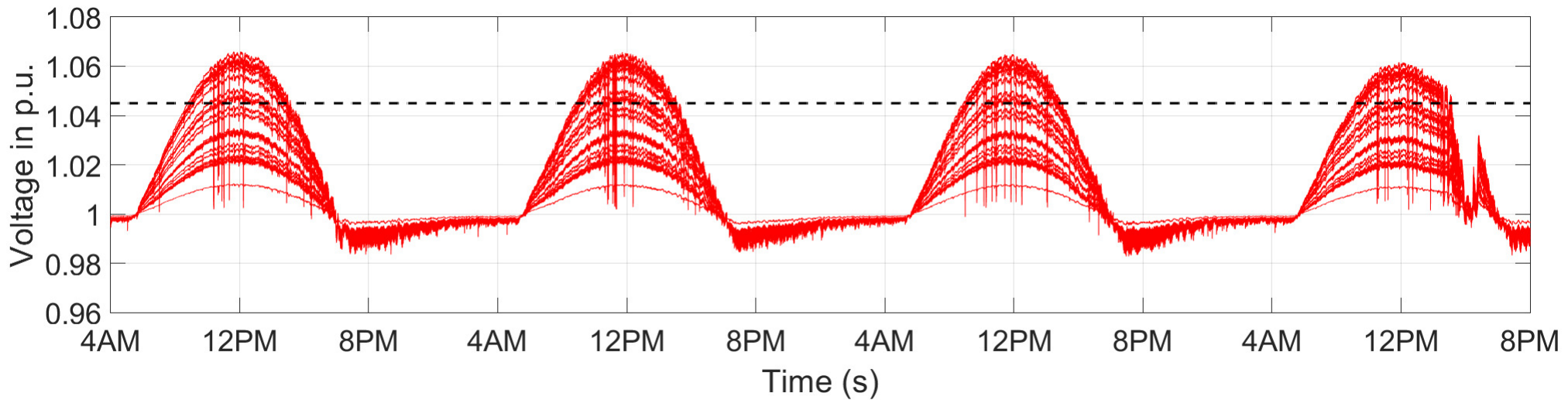}
    \caption{Uncontrolled over-voltage situation from 4:00 AM in day 1 to 8:00 PM in day 4.}
    \label{fig:uncontrol_v}
\end{figure*}
The real-time voltage trajectories of the system over four days without voltage regulation is given in Fig.~\ref{fig:uncontrol_v}. The voltage limits are set to 1.045 and 0.95 p.u. This particular network has a significant over-voltage situation if no voltage regulation is enforced. We implement the proposed online joint OPF-SE Algorithm \ref{algorithm:online_OPF-SE algorithm} to achieve the voltage regulation over time, with the objective to minimize the amount of the real power curtailment and the reactive power compensation from the available PVs (i.e., $C_t^{\textrm{OPF}}(\mathbf{p}_t,\mathbf{q}_t) = \|\mathbf{p}_{t,\textrm{av}} - \mathbf{p}_{t}\|_2^2 + 3\|\mathbf{q}_{t}\|^2_2$). For the inputs of the SE problem, the measurements within the OPF-SE algorithm include the voltage magnitude measurements at nodes 6, 7, and 24, as shown in Fig.~\ref{fig:37node}, with the measurement noise subjected to a Gaussian distribution with zero mean and 1\% standard deviation. Additionally, we also include the pseudo-measurements for all nodal injections (i.e., active and reactive power) with significant noise (e.g., zero mean and 50\% standard deviation of real values), which guarantees the full observability of the SE problem. The voltage information of the whole network is instantaneously fed back to the online OPF gradients via dual variable updates every second based on the online estimation results. 
% The estimated active and reactive power in distribution networks are the primary estimation variables. The estimated voltage magnitudes are determined by the primary estimation via linear power flow.
The gradient step sizes are set to $\alpha_{\textrm{OPF}} = 8\times 10^{-4}$, $\alpha_{\textrm{SE}} = 9\times10^{-4}$ and $\alpha_{\tau} = 3\times10^{-3}$ for primal updates and $\alpha_{\textrm{OPF}} = 5\times 10^{-3}$ for dual updates. The simulation takes $1.3030\times 10^4$ seconds to perform the proposed online joint stochastic OPF-SE algorithm for 318,800 updates (nearly 4 days) in 1-second resolution via the MATLAB interface with MATPOWER \cite{zimmerman2011matpower} on a laptop with 16GM of memory and a 2.8GHz Intel Core i7. One OPF-SE joint update takes 0.0412 seconds (on average) in computation time for every 1-second update in practice. This also includes the time for solving the nonlinear power flow via MATPOWER to simulate the response of the distribution networks to the online DER dispatch decisions. \\

% \color{blue} The simulation takes $1.3030\times 10^4$ seconds to perform the proposed online joint stochastic OPF-SE algorithm for 318,800 updates (nearly 4 days) in 1-second resolution via the MATLAB interface with MATPOWER \cite{zimmerman2011matpower} on a laptop with 16GM of memory and a 2.8GHz Intel Core i7. One OPF-SE joint update takes 0.0412 seconds (in average) computation time for every 1-second update in practice. This also includes the time for solving nonlinear power flow via MATPOWER after every online DER decisions dispatched for simulating the response of distribution networks.\color{black}

\begin{figure*}[!ht]
\centering
% controlled voltage deterministic
\subfigure[Voltage trajectories for a system controlled by the OPF-SE synthesis algorithm w/o realization of SE \& linearization errors, from 4:00 AM in day 1 to 8:00 PM in day 4.]{\label{fig:control_v_DT} %% label for first subfigure
\includegraphics[width=7.1in]{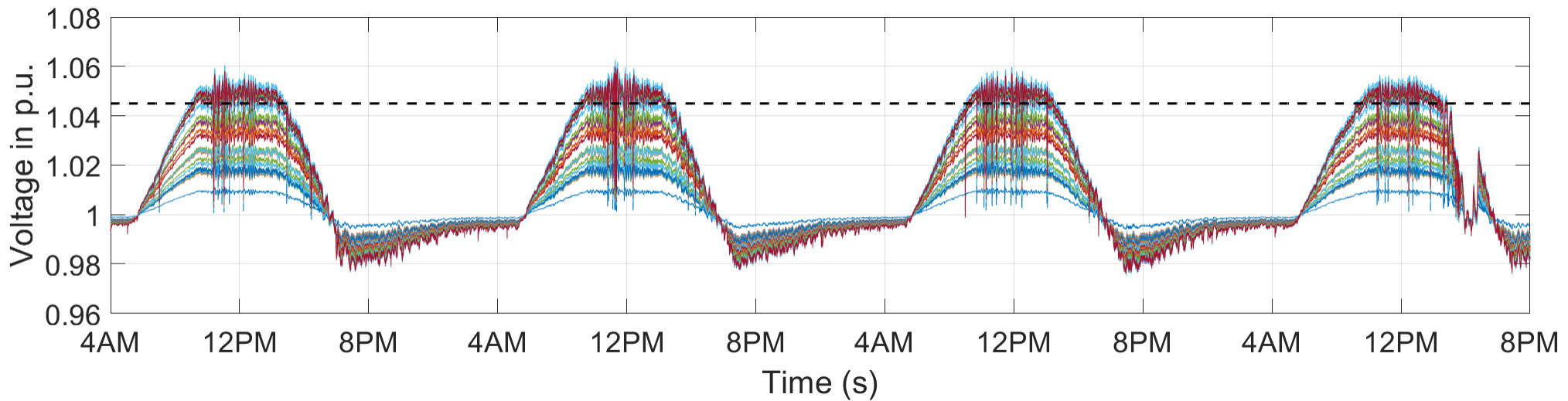}}
% \includegraphics[width=6.8in]{images/se.eps}}

% controlled voltage DT beta = 0.01
\subfigure[Voltage trajectories for a system controlled by the OPF-SE synthesis algorithm w/ realization of SE \& linearization errors. The chance constraint parameter is set to $\beta = 0.10$ with constraint satisfaction probability of 90\%, from 4:00 AM in day 1 to 8:00 PM in day 4.]{\label{fig:control_v_beta_001} %% label for first subfigure
\includegraphics[width=7.1in]{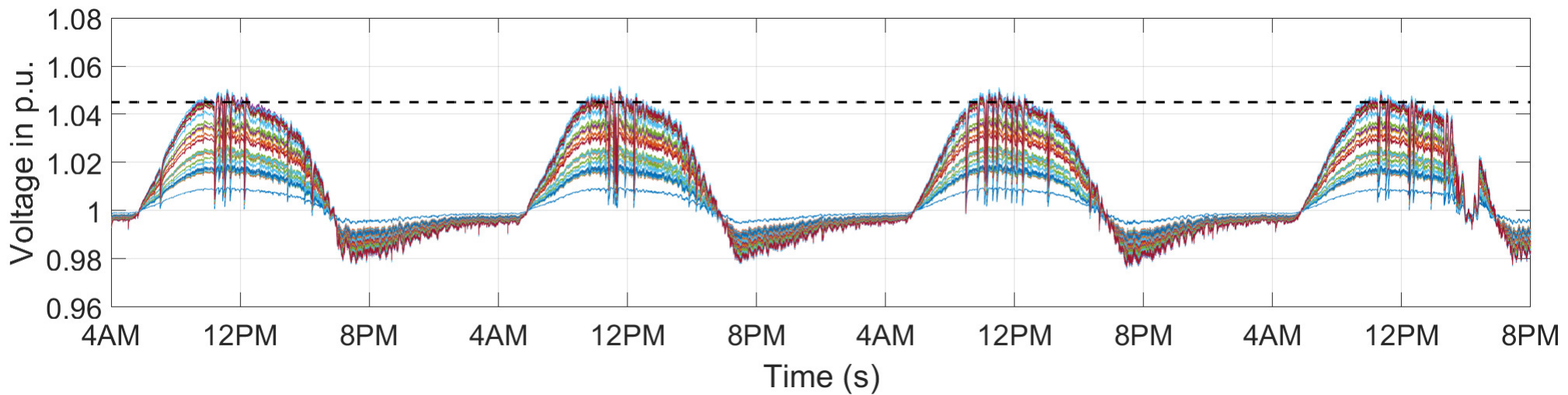}}
% controlled voltage DT beta = 0.05
\subfigure[Voltage trajectories for a system controlled by the OPF-SE synthesis algorithm w/ realization of SE \& linearization errors. The chance constraint parameter is set to $\beta = 0.05$ with constraint satisfaction probability of 95\%, from 4:00 AM in day 1 to 8:00 PM in day 4.]{\label{fig:control_v_beta_005} %% label for first subfigure
\includegraphics[width=7.1in]{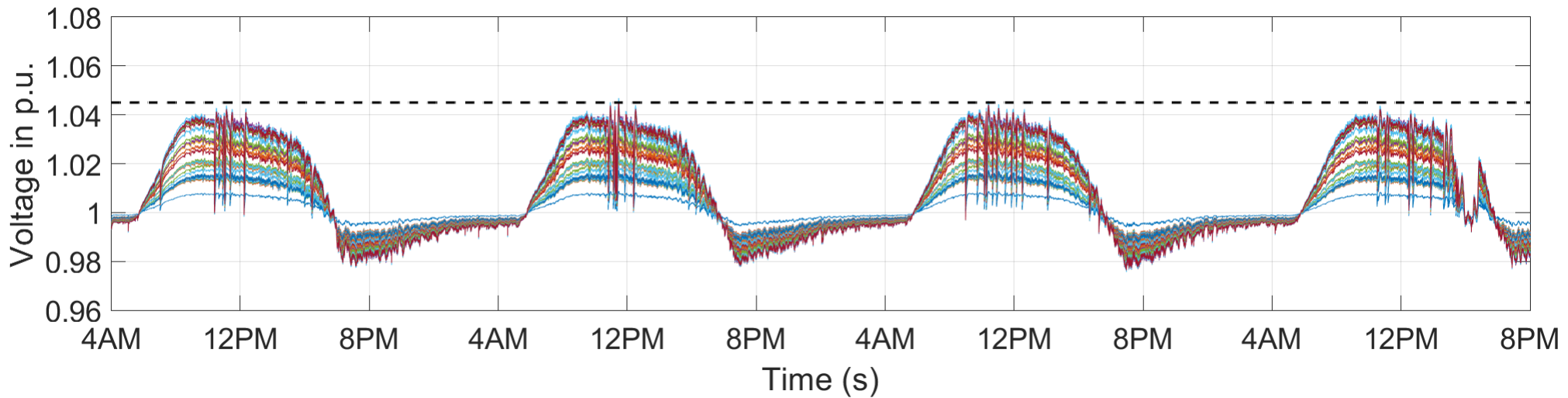}}
% controlled voltage DT beta = 0.05
\subfigure[Voltage trajectories for a system controlled by the OPF-SE synthesis algorithm with realization of SE \& linearization errors. The chance constraint parameter is set to $\beta = 0.01$ with constraint satisfaction probability of 99\%, from 4:00 AM in day 1 to 8:00 PM in day 4.]{\label{fig:control_v_beta_010} %% label for first subfigure
\includegraphics[width=7.1in]{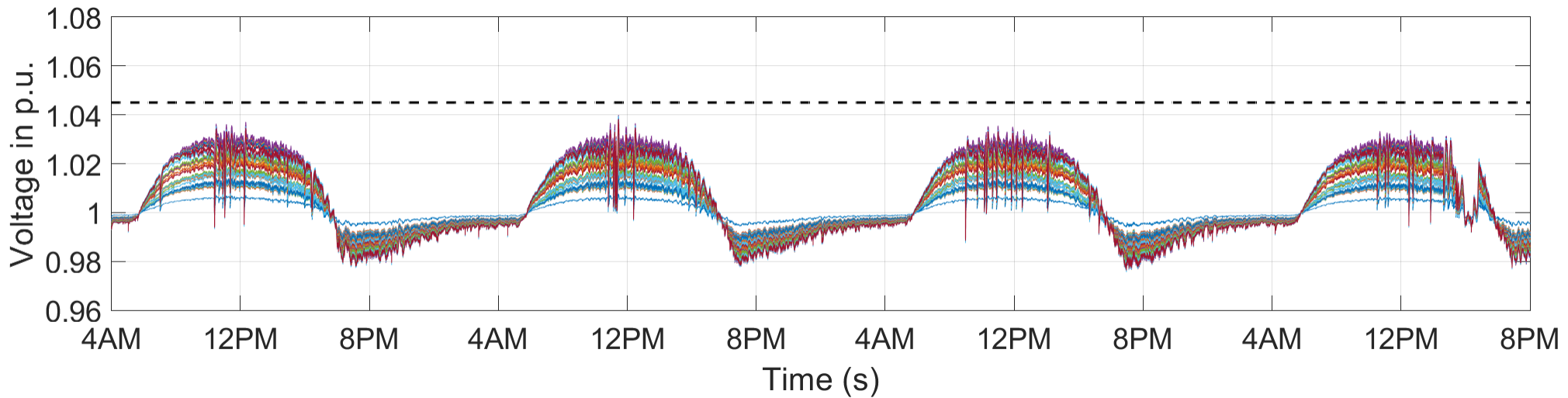}}
% having caption here
\caption{Comparison of voltage trajectories for various values of chance constraint risk aversion $\beta$. As these parameters decrease, more active power from PVs is curtailed, leading to a lower risk of voltage constraints violation.}
\label{fig:control_v} %% label for entire figure
\end{figure*}

Fig.~\ref{fig:control_v} visualizes the online voltage regulation trajectories regulated by the proposed online joint OPF-SE algorithm in \eqref{eq:stochastic_opf_se_controller}. In case of the stochastic OPF-SE, the curves for different settings of the chance constraint satisfaction are given. In order to prevent voltage from raising above 1.045 p.u., the online OPF decisions must curtail excessive PV injections given the online SE results over time. In Fig.~\ref{fig:control_v_DT}, the voltage trajectories slightly violate the upper bounds, which is because the deterministic version does not take into account that the voltage estimation and power flow linearization induce errors. The proposed stochastic OPF-SE given in Algorithm 1 provides better robustness to errors from voltage estimation and linearization. The number of samples within the dataset $\Xi_s$ is 100, generated by a Gaussian distribution with zero mean and 1\% standard deviation of the real voltage, which includes the estimation noises and the deviation caused by the linearization of the nonlinear power flow. In practice, the dataset of these errors can be collected through the observation and comparison between the historical records and estimation results.

Figs.~\ref{fig:control_v_beta_001} - \ref{fig:control_v_beta_010} illustrate the online voltage trajectories for varying chance constraint risk aversion (i.e., $\beta = 0.10, 0.05$ and $0.01$). The stochastic formulation accounts for the errors and tradeoff the CVaR of voltage constraint violation and operational cost during the online adjustment of PV injections. The conservativeness of the voltage regulation is controlled by adjusting $\beta$. By explicitly using the sampling data to take into account the inherent errors, the voltage violation risk can be systematically assessed and controlled. The curtailment decisions can be over-conservative when we require higher probability satisfaction for the voltage constraints, as shown in Figs.~\ref{fig:control_v_beta_005} and \ref{fig:control_v_beta_010}. In general, it is possible to prioritize the voltage regulation at certain buses by adapting the corresponding risk aversion, which depends on the risk preference of system operators and the accuracy performance of the estimators. Note that the PV outputs in day 3 and day 4 have large fluctuations. The proposed online joint OPF-SE algorithm however has superior robust performance and provides fast tracking in response to these large variants by generating timely optimization-estimation results. Hence, the voltage regulation can be achieved in an online fashion under a time-varying setting with large variations of PV outputs.

We now compare our OPF-SE approach to an online feedback-based OPF solver with 1) perfect information of all the voltage magnitudes \cite{dall2016optimal} and 2) all the raw measurements of voltage magnitudes, i.e. including noise. The settings are the same as the online OPF-SE algorithm except that the online OPF solver utilizes the direct (or noisy) voltage measurements. As shown in Fig.~\ref{fig:controlled_v_OPFpursuit_perfect}, incorporating full and perfect voltage feedback information into the online OPF solver results in a better profile with only very few violations compared to Fig.~\ref{fig:control_v_DT}. However, in practice, real-time voltage measurements are inherently noisy. Hence, we subject the actual voltage magnitude values to independent Gaussian distributions with zero mean and 1\% standard deviation of their actual values. Fig.~\ref{fig:controlled_v_OPFpursuit_noisy} gives the results for the online feedback-based OPF algorithm with raw voltage magnitude measurements. Due to the inherent sensing noise, the online feedback OPF solver fails to resolve the over-voltage situation in a fast-changing distribution network. Clearly, having an unbiased WLS state estimator can significantly reduce the uncertainties in the feedback loop and promotes the feasibility of the online OPF solver. This indicates that our proposed online OPF-SE algorithm having available pseudo-measurements and a limited number of raw voltage measurements has superior robust performance, compared to the direct usage of all the raw measurements of voltage magnitudes.\\ 

  \begin{figure*}[h]
    \centering
    \includegraphics[width=7.1in]{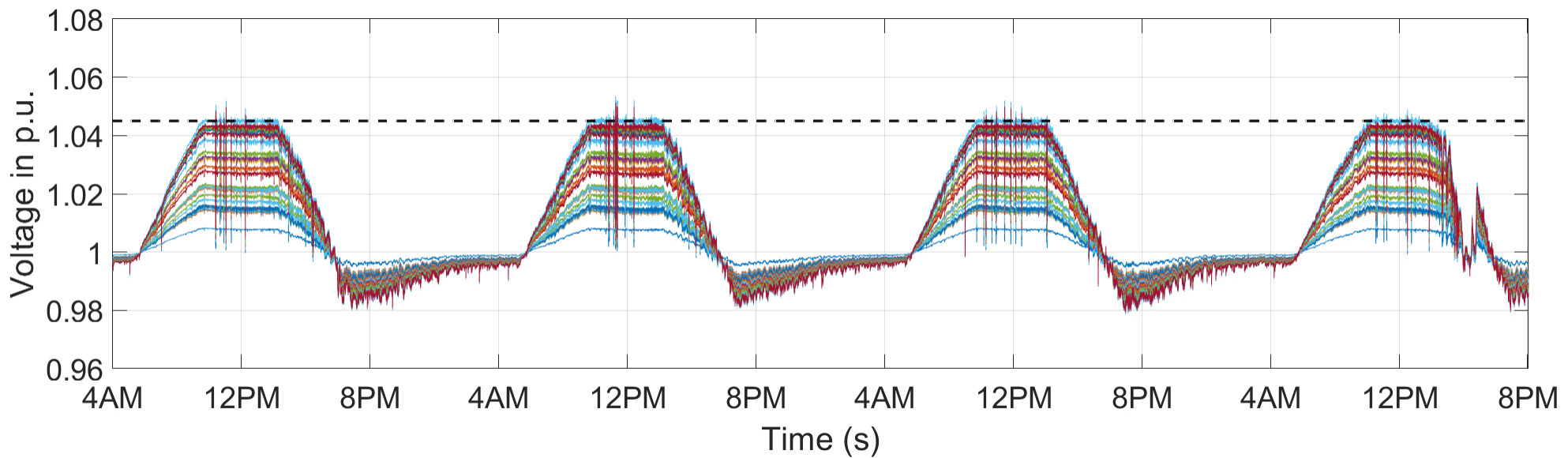}
    \caption{Voltage trajectories for a system controlled by an online feedback-based OPF solver with perfect information of all voltage magnitudes.}
    \label{fig:controlled_v_OPFpursuit_perfect}
    \end{figure*}

 \begin{figure*}[h]
    \centering
    \includegraphics[width=7.1in]{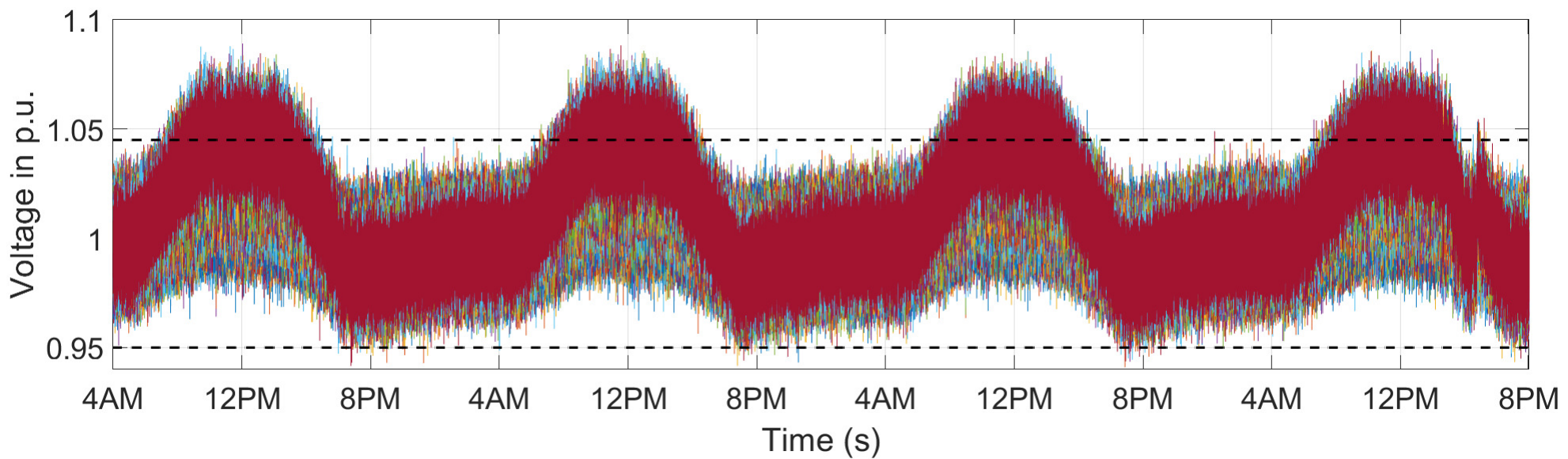}
    \caption{Voltage trajectories for a system controlled by an online feedback-based OPF solver using raw measurements of all voltage magnitudes.}
    \label{fig:controlled_v_OPFpursuit_noisy}
\end{figure*}
% In the mean time, the time-varying voltage profiles are tracked timely as well in Fig.~\ref{fig:se_v_ST_beta005}. The online dual variable trajectories in Fig.~\ref{fig:lambda_v_ST_beta005} indicate the over-voltage situation based on the SE tracking Fig.~\ref{fig:se_v_ST_beta005}, in turn adjust the set-points of DERs to mitigate the over-voltage over time. Note that online optimal dual solutions are robust enough to the large variations of PV injections in a fast resolution, i.e., 1-second.
Overall, we conclude that the proposed online optimization-estimation architecture is capable to achieve online voltage regulation under a time-varying situation. By introducing the stochastic modelling of inherent noises within the feedback loop, the online regulation enables robust performance with respect to the estimation/linearization errors, which provides operators options to run the systems under various risk aversion settings. The benefit of running optimal control-estimation synthesis has been shown from the perspectives of robustness and computational efficiency.

% \begin{figure*}[!ht]
%     \centering
%     \includegraphics[width=7.0in]{images/se_v_ST_beta005.eps}
%     \caption{Estimated voltage trajectories from OPF-SE synthesis. The chance constraint parameter is set to $\beta
%     = 0.05$ with constraint satisfaction in probability of 95\%, from 4:00 AM in day 1 to 8:00 PM in day 4.}
%     \label{fig:se_v_ST_beta005}
% \end{figure*}

% \begin{figure*}[!ht]
%     \centering
%     \includegraphics[width=6.5in]{images/lambda_v_ST_beta005.eps}
%     \caption{The online trajectories of dual variables associated with the voltage upper bound. The chance constraint parameter is set to $\beta = 0.05$ with constraint satisfaction in probability of 95\%, from 4:00 AM in day 1 to 8:00 PM in day 4.}
%     \label{fig:lambda_v_ST_beta005}
% \end{figure*}

\section{Conclusions}\label{sec:conclusions}
In this paper, we provide an extensive theoretical analysis and numerical results of an online joint optimization-estimation architecture for distribution networks. An online implementation is proposed to solve a time-varying OPF problem and a WLS SE problem in parallel with a large penetration of renewable penetration. The online stochastic framework explicitly considers the SE and AC power flow linearization errors for robust performance. Convergence and optimality of the proposed algorithm are analytically established. The numerical results demonstrate the necessity and success of bridging the traditional gap between optimization and estimation layers in distribution networks, from the perspectives of computational efficiency, robustness, effectiveness and flexibility.

\bibliographystyle{ieeetr}  
\bibliography{reference} 

% \clearpage
% \clearpage
% \newpage
% \newpage
% \clearpage
\appendix
\subsection{Proof of Lemma \ref{lemma:strongly_monotone} }\label{appendix_lemma_monotone}
\begin{proof}
We decompose the gradient operator $\mathcal{F}_t(\cdot)$ in \eqref{eq:F_operator} equivalently to:
% \color{blue}
% \begin{equation}\nonumber
% \begin{aligned}
%     & \mathcal{F}_t(\mathbf{u}_t,\bm{\mu}_t,\mathbf{z}_t) = \\
%     & ~~~~~~~~~~ \begin{bmatrix}
%     \nabla_{\mathbf{u}}C^{\textrm{OPF}}_t(\mathbf{u}_t)\\
%     \nabla_{\bm{\mu}}\frac{\phi}{2}\|\bm{\mu}_t\|_2^2\\
%     \nabla_{\mathbf{z}}C^{\textrm{SE}}_t(\mathbf{z}_t)
%     \end{bmatrix} + \begin{bmatrix}
%      0 & -\mathbf{H}^\top & 0\\
%      \mathbf{H} & 0 & 0\\
%      0 & 0 & 0
%      \end{bmatrix}\begin{bmatrix}
%      \mathbf{u}_t\\
%      \bm{\mu}_t\\
%      \mathbf{z}_t
%      \end{bmatrix} + \mathbf{a}_0,
% \end{aligned}
% \end{equation}\color{black}
\begin{equation}\nonumber
\begin{aligned}
    & \mathcal{F}_t(\mathbf{u}_t,\bm{\mu}_t,\mathbf{z}_t) = \\
    & ~~~~~~~~~~ \begin{bmatrix}
    \nabla_{\mathbf{u}}C^{\textrm{OPF}}_t(\mathbf{u}_t)\\
    \nabla_{\mathbf{z}}C^{\textrm{SE}}_t(\mathbf{z}_t)\\
    \nabla_{\bm{\mu}}\frac{\phi}{2}\|\bm{\mu}_t\|_2^2
    \end{bmatrix} + \begin{bmatrix}
     0 & 0 & -\mathbf{H}^\top\\
     0 & 0 & 0\\
     0 & \mathbf{H} & 0
     \end{bmatrix}\begin{bmatrix}
     \mathbf{u}_t\\
     \mathbf{z}_t\\
     \bm{\mu}_t
     \end{bmatrix} + \mathbf{a}_0,
\end{aligned}
\end{equation}
where $\mathbf{H} = \begin{bmatrix} \mathbf{R} & \mathbf{X}\\ -\mathbf{R} & -\mathbf{X} \end{bmatrix}$ and $\mathbf{a}_0$ denotes a constant vector. Since the estimation feedback $\mathbf{v}_t^k(\mathbf{z}_t^k) = \mathbf{v}_t^k(\mathbf{u}_t^k)$ always holds every gradient step, the gradient operator can be re-written by replacing $\mathbf{v}_t^k(\mathbf{z}_t^k)$ by $\mathbf{v}_t^k(\mathbf{u}_t^k)$:
\begin{equation}\nonumber
\begin{aligned}
    &\mathcal{F}_t(\mathbf{u}_t,\bm{\mu}_t,\mathbf{z}_t) =\\
    & ~~~~~~~~~~ \begin{bmatrix}
    \nabla_{\mathbf{u}}C^{\textrm{OPF}}_t(\mathbf{u}_t)\\
    \nabla_{\mathbf{z}}C^{\textrm{SE}}_t(\mathbf{z}_t)\\
    \nabla_{\bm{\mu}}\frac{\phi}{2}\|\bm{\mu}_t\|_2^2
    \end{bmatrix} + \underbrace{\begin{bmatrix}
     0 & 0 & -\mathbf{H}^\top \\
     0 & 0 & 0\\
     \mathbf{H} & 0 & 0
     \end{bmatrix}}_{**}\begin{bmatrix}
     \mathbf{u}_t\\
     \mathbf{z}_t\\
     \bm{\mu}_t
     \end{bmatrix} + \mathbf{a}_0,
\end{aligned}
\end{equation}
This rewriting implies that the online dual update depends on the online SE results instead of OPF decisions at every iteration. The first term in the above decomposition is strongly monotone, since each element is strongly convex in $(\mathbf{u}_t,\bm{\mu}_t,\mathbf{z}_t)$. We observe that the second linear operator with the projection matrix $**$ is monotone if and only if the following condition holds \cite{ryu2022large}:
\begin{equation*}
\begin{bmatrix}
     0 & 0 & -\mathbf{H}^\top \\
     0 & 0 & 0\\
     \mathbf{H} & 0 & 0
     \end{bmatrix} + \begin{bmatrix}
     0 & 0 & -\mathbf{H}^\top \\
     0 & 0 & 0\\
     \mathbf{H} & 0 & 0
     \end{bmatrix}^\top \succeq 0.
\end{equation*}
The above condition holds by the definition of positive semi-definite matrix. In the end, the gradient operator $\mathcal{F}_t(\cdot)$ is strongly monotone due to a linear combination of a strong monotone operator and a monotone operator, which concludes the proof.
\end{proof}

\subsection{Proof of Theorem \ref{thm:reverse_engineering}}\label{appendix_reverse_engineering}
\begin{proof}
Introducing a Lagrangian multiplier $\bm{\lambda}_t \in\mathbb{R}^{2N}$ for the constraint \eqref{eq:p3_voltage_constraint_upper}, we have the following regularized Lagrangian with the parameter $\phi > 0$,
\begin{equation}\label{eq:L_opf_se}
\begin{aligned}
& \widetilde{\mathcal{L}}_t\left(\mathbf{u}_t,\mathbf{z}_t,\bm{\lambda}_t \right) = \\
& C_t^{\textrm{OPF}}(\mathbf{u}_t) + C^{\textrm{SE}}_t(\mathbf{z}_t,\mathbf{v}_t(\mathbf{z}_t)) + 
\bm{\lambda}_t^\top \mathbf{r}(\mathbf{v}_t(\mathbf{u}_t)) - \frac{\phi}{2}\|\bm{\lambda}_t\|_2^2.
\end{aligned}
\end{equation}
We have a saddle-point problem:
\begin{equation}\label{eq:saddle_opf_se}
    \max_{\bm{\lambda}_t\in\mathbb{R}_+} \min_{\mathbf{u}_t\in\mathcal{X}_t,\mathbf{z}_t\in\mathcal{X}_t} \widetilde{\mathcal{L}}_t(\mathbf{u}_t,\mathbf{z}_t,\bm{\lambda}_t).
\end{equation}
Denoting $(\mathbf{u}_t^*, \mathbf{z}_t^*,\bm{\lambda}_t^*)$ as a unique primal-dual optimizer of \eqref{eq:saddle_opf_se} at time $t$, the optimality conditions of \eqref{eq:saddle_opf_se} are:
\begin{subequations} \label{eq:opf_se_condition}
\begin{align}
&\left(\nabla_{\mathbf{u}}C^{\textrm{OPF}}_t(\mathbf{u}_t^*) -\nabla_{\mathbf{u}}\mathbf{r}(\mathbf{v}_t(\mathbf{u}_t^*))^\top \bm{\mu}_t^*\right)^\top \left(\mathbf{u}_t - \mathbf{u}_t^*\right)\geq 0,\nonumber\\ & ~~~~~~~~~~~~~~~~~~~~~~~~~~~~~~~~~~~~~~~~~~~~~~~~~~~~\mathbf{u}_t\in\mathcal{X}_t, \label{eq:cond_u}\\
&\left(\mathbf{r}(\mathbf{v}_t(\mathbf{z}_t^*))- \phi\bm{\lambda}^*_t\right)^\top \bm{\lambda}_t^* = 0,\label{eq:cond_dual1}\\
&\bm{\lambda}_t^* \geq 0, \label{eq:cond_lambda}\\
&\left(\nabla_{\mathbf{z}}C^{\textrm{SE}}_t(\mathbf{z}_t,\mathbf{v}_t(\mathbf{z}_t))\right)^\top \left(\mathbf{z}_t - \mathbf{z}_t^*\right)\geq 0, \label{eq:cond_z}
\end{align}
\end{subequations}
which are equivalent to:
\begin{subequations}\label{eq:optimality_joint_opf_se}
\begin{align}
\mathbf{u}_t^{*} & =  \Bigg[\mathbf{u}_t^* - \epsilon\bigg( \nabla_{\mathbf{u}}C_t^{\textrm{OPF}}(\mathbf{u}_t^*) + \nabla_{\mathbf{u}}\mathbf{r}(\mathbf{v}_t(\mathbf{\mathbf{u}}_t^*))^\top\bm{\mu}_t^*\bigg)\Bigg]_{\mathcal{X}_t},\\
\bm{\mu}_t^{*} & = \Bigg[ \bm{\mu}_t^* + \epsilon\bigg(\mathbf{r}(\mathbf{v}_t(\mathbf{z}_t^*)) - \phi\bm{\mu}_t^*\bigg) \Bigg]_{\mathbb{R}_+},\\
\mathbf{z}_t^{*} & =  \mathbf{z}_t^* - \epsilon \nabla_{\mathbf{z}}C_t^{\textrm{SE}}(\mathbf{z}_t^*,\mathbf{v}_t(\mathbf{z}_t^*)).
\end{align}
\end{subequations}
% \begin{subequations}
% \begin{align}
% \mathbf{p}_t^* & =  \Bigg[\mathbf{p}_t^* - \epsilon\bigg( \nabla_{\mathbf{p}}C_t^{\textrm{OPF}}(\mathbf{p}_t^*,\mathbf{q}_t^*) +\mathbf{R}^\top\left(\overline{\bm{\mu}}_t^* - \underline{\bm{\mu}}_t^*\right)\bigg)\Bigg]_{\mathcal{X}_t}, \\
% \mathbf{q}_t^* & =  \Bigg[\mathbf{q}_t^* - \epsilon\bigg( \nabla_{\mathbf{q}}C^{\textrm{OPF}}_t(\mathbf{p}_t^*,\mathbf{q}_t^*) +\mathbf{X}^\top\left(\overline{\bm{\mu}}^*_t - \underline{\bm{\mu}}^*_t\right)\bigg)\Bigg]_{\mathcal{X}_t}, \\
% \tilde{\mathbf{p}}^*_t & =  \left[\tilde{\mathbf{p}}^*_t - \epsilon\left( \nabla_{\tilde{\mathbf{p}}}C^{\textrm{SE}}_t(\tilde{\mathbf{p}}_t^*,\tilde{\mathbf{q}}_t^*)\right)\right]_{\mathcal{X}_t}, \\
% \tilde{\mathbf{q}}_t^* & =  \left[\tilde{\mathbf{q}}_t^* - \epsilon\left( \nabla_{\tilde{\mathbf{q}}}C_t^{\textrm{SE}}(\tilde{\mathbf{p}}_t^*,\tilde{\mathbf{q}}_t^*)\right)\right]_{\mathcal{X}_t}, \\
% \underline{\bm{\mu}}_t^* & = \left[ \underline{\bm{\mu}}_t^* + \epsilon\left(\underline{\mathbf{v}} - \mathbf{d}(\mathbf{p}^*_t,\mathbf{q}^*_t) - \eta\underline{\bm{\mu}}_t^*\right) \right]_{\mathbb{R}_+}\\
% \overline{\bm{\mu}}_t^{*} & = \left[\overline{\bm{\mu}}_t^* + \epsilon\left( \mathbf{d}(\mathbf{p}^*_t,\mathbf{q}^*_t) - \overline{\mathbf{v}}- \eta\overline{\bm{\mu}}^t_*\right)\right]_{\mathbb{R}_+}.
% \end{align}
% \end{subequations}
The point $(\mathbf{u}_t^*, \bm{\lambda}_t^*, \mathbf{z}_t^*)$ in \eqref{eq:opf_se_condition} is a unique saddle-point of \eqref{eq:saddle_opf_se} if and only if the point $(\mathbf{u}_t^*,\bm{\mu}_t^*,\mathbf{z}_t^*)$ in \eqref{eq:optimality_joint_opf_se} is an approximate solution of \eqref{eq:opf_se} and the estimation feedback $\mathbf{v}_t^{k}(\mathbf{u}_t^k) = \mathbf{v}_t(\mathbf{z}_t^k)$ holds for all iterations in \eqref{eq:algorithm_joint_t}. This also implies that the optimal voltage profiles are equivalent to the estimation results, such that $\mathbf{v}_t(\mathbf{u}_t^*) = \mathbf{v}_t(\mathbf{z}_t^*)$. In addition of $\bm{\mu}_t^* = \bm{\lambda}_t^*$, the existence and uniqueness of the saddle-point $(\mathbf{u}_t^*,\bm{\lambda}_t^*,\mathbf{z}_t^*)$ of \eqref{eq:saddle_opf_se} then implies the equilibrium $(\mathbf{u}_t^*,\bm{\mu}_t^*,\mathbf{z}_t^*)$ of \eqref{eq:algorithm_joint_t}, which concludes the proof.
\end{proof}

\subsection{Proof of Theorem \ref{thm:performance_bound}}\label{appendix_performance_bound}
\begin{proof}
In this analysis, we omit the time-index subscript $t$ to simplify the notation. We relax $\|\mathbf{x}^* - \mathbf{x}_{\phi,v}^*\|^2_2$ into two parts via the triangle inequality:
\begin{equation}\label{eq:bound_relax}
\|\mathbf{x}^* - \mathbf{x}_{\eta}^*\|^2_2 \leq \|\mathbf{x}^* - \mathbf{x}_{v}^*\|^2_2 +  \|\mathbf{x}^*_v - \mathbf{x}_{\eta}^*\|^2_2,
\end{equation}
and characterize the upper bounds of the two terms on the right-hand side of \eqref{eq:bound_relax} as follows.

1) We first derive the bound of $\|\mathbf{x}^* - \mathbf{x}^*_v\|_2^2$. Define a vector to collect the subset of primal variables in $\tilde{\mathbf{x}}:=[\mathbf{u}^\top,\mathbf{z}^\top]^\top$ and a compact objective expression of \eqref{eq:opf_se_cvar_objective} as $f(\tilde{\mathbf{x}}):= C^{\textrm{OPF}}(\tilde{\mathbf{x}}) + C^{\textrm{SE}}(\tilde{\mathbf{x}})$.
By viewing the regularized Lagrangian $\mathcal{L}_\eta$ in \eqref{eq:stochastic_L_opf_se_phi_v} as a result of a two-step regularization, we first regularize the Lagrangian function \eqref{eq:L_stochastic_opf_se_no-regularizer} of the original problem \eqref{eq:opf_se_cvar} by adding an additional regularization term $\frac{v}{2}\|\bm{\tau}\|^2_2$ on the primal variable $\bm{\tau}$ with a constant $v>0$, i.e.,
\begin{equation}\label{eq:Lv}\nonumber
    \mathcal{L}_{v}(\mathbf{x},\bm{\lambda}) = f(\tilde{\mathbf{x}}) + \bm{\lambda}^\top \mathbf{g}(\mathbf{x}) + \frac{v}{2}\|\bm{\tau}\|^2_2,
\end{equation}
and its corresponding saddle-point problem is given by
\begin{equation}\label{eq:saddle_point_problem_regularized_v}
    \max_{\bm{\lambda}\in\mathbb{R}^N_{+}}\min_{\mathbf{x}\in\mathcal{X}\times\mathcal{X}\times\mathbb{R}^{2N}} \mathcal{L}_{v} (\mathbf{x},\bm{\lambda}).
\end{equation}
The saddle-point $(\mathbf{x}_v^*, \bm{\lambda}_v^*)$ of \eqref{eq:saddle_point_problem_regularized_v} follows the property of
\begin{equation}
    \mathcal{L}_{v}(\mathbf{x}_v^*,\bm{\lambda}) \leq \mathcal{L}_{v}(\mathbf{x}_v^*,\bm{\lambda}_v^*) \leq \mathcal{L}_{v}(\mathbf{x},\bm{\lambda}_v^*), \quad \forall \mathbf{x}, \bm{\lambda}.
\end{equation}
Let $\mathbf{x} = \mathbf{x}^*$ in the second inequality of the preceding relationship, we obtain:
\begin{equation}
\begin{aligned}
& f(\tilde{\mathbf{x}}_v^*) + (\bm{\lambda}_v^*)^\top \mathbf{g}(\mathbf{x}_v^*) + \frac{v}{2}\|\bm{\tau}_v^*\|^2_2 \\ 
& ~~~~~~~~~~~~~~~~~~~~\leq f(\tilde{\mathbf{x}}^*) + (\bm{\lambda}_v^*)^\top \mathbf{g}(\mathbf{x}^*) + \frac{v}{2}\|\bm{\tau}^*\|^2_2.
\end{aligned}
\end{equation}
With functions $\tilde{f}(\mathbf{x}^*):=f(\tilde{\mathbf{x}}^*) + \frac{v}{2}\|\bm{\tau}^*\|^2_2$ and $\tilde{f}(\mathbf{x}^*_v):=f(\tilde{\mathbf{x}}^*_v) + \frac{v}{2}\|\bm{\tau}^*_v\|^2_2$, the above inequality becomes:
\begin{equation}
\begin{aligned}
& \tilde{f}(\mathbf{x}_v^*) + (\bm{\lambda}_v^*)^\top \mathbf{g}(\mathbf{x}_v^*) \leq \tilde{f}(\mathbf{x}^*) + (\bm{\lambda}_v^*)^\top \mathbf{g}(\mathbf{x}^*),
\end{aligned}
\end{equation}
which then leads to:
\begin{equation}\label{eq:proof_f_difference_1}
\begin{aligned}
& \tilde{f}(\mathbf{x}_v^*) - \tilde{f}(\mathbf{x}^*) \leq (\bm{\lambda}_v^*)^\top \mathbf{g}(\mathbf{x}^*) - (\bm{\lambda}_v^*)^\top \mathbf{g}(\mathbf{x}_v^*).
\end{aligned}
\end{equation}
Observe that the newly defined function $\tilde{f}(\cdot)$ is strongly convex in all primal variables $\mathbf{x}$ for some $c>0$,
\begin{equation}\label{eq:f_strong_convexity_1}
  \tilde{f}(\mathbf{x}_v^*) - \tilde{f}(\mathbf{x}^*) \geq \nabla_\mathbf{x}\tilde{f}^\top (\mathbf{x}^*) \left(\mathbf{x}_v^* - \mathbf{x}^*\right) + \frac{c}{2}\|\mathbf{x}_v^* - \mathbf{x}^*\|^2_2. 
\end{equation}
Combining \eqref{eq:f_strong_convexity_1} and \eqref{eq:proof_f_difference_1}, we obtain:
\begin{equation}\nonumber
\begin{aligned}
    &\frac{c}{2}\|\mathbf{x}_v^* - \mathbf{x}^*\|^2_2 \\
    &~~\leq \nabla_\mathbf{x} \tilde{f}^\top(\mathbf{x}^*)\left(\mathbf{x}^* - \mathbf{x}^*_v\right) + \left(\bm{\lambda}_v^*\right)^\top \left(\mathbf{g}(\mathbf{x}^*) - \mathbf{g}(\mathbf{x}_v^*)\right).
\end{aligned}
\end{equation}
Due to the convexity of $\mathbf{g}(\cdot)$, we can write the second multiplication term over all $j$ and add them up,
\begin{equation}\nonumber
\begin{aligned}
&\frac{c}{2}\|\mathbf{x}_v^* - \mathbf{x}^*\|^2_2 \\
    &\leq \nabla_\mathbf{x} \tilde{f}^\top(\mathbf{x}^*)\left(\mathbf{x}^* - \mathbf{x}^*_v\right) + \sum_{j=1}^{2N}\bm{\lambda}_{j,v}^*\nabla_\mathbf{x} \mathbf{g}^\top_j (\mathbf{x}^*)(\mathbf{x}^* - \mathbf{x}^*_v) \\
    &\leq \left\|\nabla_\mathbf{x} \tilde{f}^\top(\mathbf{x}^*)\left(\mathbf{x}^* - \mathbf{x}^*_v\right)\right\|_1 \\
    & ~~~+ \sum_{j=1}^{2N}\bm{\lambda}_{j,v}^*\left\|\nabla_\mathbf{x} \mathbf{g}^\top_j (\mathbf{x}^*)(\mathbf{x}^* - \mathbf{x}^*_v)\right\|_1.
\end{aligned}
\end{equation}  
Employing the Cauchy–Schwarz inequality, the second inequality above results in:
\begin{equation}\nonumber
\begin{aligned}
    &\frac{c}{2}\|\mathbf{x}_v^* - \mathbf{x}^*\|^2_2 \leq \|\nabla_\mathbf{x} \tilde{f}(\mathbf{x}^*)\|_2 \|\mathbf{x}^* - \mathbf{x}^*_v\|_2 \\
    & ~~~~~~~~~~~~~~~~~~~~~+ \sum_{j}\bm{\lambda}_{j,v}^*\|\nabla_\mathbf{x} \mathbf{g}_j (\mathbf{x}^*)\|_2\|\mathbf{x}^* - \mathbf{x}^*_v\|_2.
\end{aligned}
\end{equation}
Since the gradients/subgradients of $\tilde{f}(\cdot)$ and $\mathbf{g}(\cdot)$ are bounded in 
Assumption \ref{ass:Lipschitz_Property_gradient_bound}, this leads to:
\begin{equation}\nonumber
\begin{aligned}
    &\frac{c}{2}\|\mathbf{x}_v^* - \mathbf{x}^*\|^2_2 \\
    &~~~~~~~~\leq G_f \|\mathbf{x}^* - \mathbf{x}^*_v\|_2 + \sum_{j}\bm{\lambda}_{j,v}^* G_g\|\mathbf{x}^* - \mathbf{x}^*_v\|_2.
\end{aligned}
\end{equation}
Next, we divide both sides of the above inequality by $\|\mathbf{x}_v^* - \mathbf{x}^*\|_2$ resulting in
\begin{equation*}
\begin{aligned}
    &\frac{c}{2}\|\mathbf{x}_v^* - \mathbf{x}^*\|_2 \leq G_f + G_g \|\bm{\lambda}_{v}^*\|_1,
\end{aligned}
\end{equation*}
and attain:
\begin{equation} \label{eq:bound_v}
\begin{aligned}
    &\|\mathbf{x}_v^* - \mathbf{x}^*\|_2 \leq 
    \frac{2\left(G_f + G_g \|\bm{\lambda}_{v}^*\|_1\right)}{c}.
\end{aligned}
\end{equation}

2) We now characterize the upper bound of $\|\mathbf{x}_{v} - \mathbf{x}_{\eta}\|_2^2$, the second term on the right-hand side of \eqref{eq:bound_relax}. Consider a regularized Lagrangian with regularization terms on both primal and dual variables, i.e.,
\begin{equation*}
    \mathcal{L}_{\eta}(\mathbf{x},\bm{\lambda}) = f(\tilde{\mathbf{x}}) + \bm{\lambda}^\top g(\mathbf{x}) + \frac{v}{2}\|\bm{\tau}\|^2_2 - \frac{\phi}{2}\|\bm{\lambda}\|^2_2,
\end{equation*}
where $v>0$ and $\phi>0$. Its corresponding saddle-point problem is:
\begin{equation}\label{eq:saddle_point_problem_eta}
    \max_{\bm{\lambda}\in\mathbb{R}_+} \min_{\mathbf{x}\in\mathcal{X}\times\mathcal{X}\times\mathbb{R}_+} \mathcal{L}_{\eta}(\mathbf{x},\bm{\lambda}).
\end{equation}
The saddle-point $(\mathbf{x}_{\eta}^*,\bm{\lambda}_{\eta}^*)$ of \eqref{eq:saddle_point_problem_eta} follows the property of
% \begin{equation*}
%     \mathcal{L}_{\eta}(\mathbf{x},\bm{\lambda}) = f(\tilde{\mathbf{x}}) + \bm{\lambda}^\top g(\mathbf{x}) + \frac{v}{2}\|\bm{\tau}\|^2_2 - \frac{\phi}{2}\|\bm{\lambda}\|^2_2,
% \end{equation*}
\begin{equation}\nonumber
    \mathcal{L}_{\eta}(\mathbf{x}_{\eta}^*,\bm{\lambda}) \leq \mathcal{L}_{\eta}(\mathbf{x}_{\eta}^*,\bm{\lambda}_{\eta}^*) \leq \mathcal{L}_{\eta}(\mathbf{x},\bm{\lambda}_{\eta}^*),~~ \forall \mathbf{x}, \bm{\lambda}.
\end{equation}
The left inequality leads to:
\begin{equation}\label{eq:proof_two_g}
    \left(\bm{\lambda}_{\eta}^* - \bm{\lambda}_{v}^*\right)^\top \mathbf{g}(\mathbf{x}_{\eta}^*) - \frac{\phi}{2}\|\bm{\lambda}_{\eta}^*\|^2_2 + \frac{\phi}{2}\|\bm{\lambda}_{v}^*\|^2_2 \geq 0,
\end{equation}
where we set $\bm{\lambda} = \bm{\lambda}_v^*$. We now characterize the term $\left(\bm{\lambda}_{\eta}^* - \bm{\lambda}_{v}^*\right)^\top \mathbf{g}(\mathbf{x}_{\eta}^*)$. Leveraging the definition of convex function, the upper bound of $\mathbf{g}_j(\mathbf{x}_{\eta}^*)$ is given by,
\begin{equation}\label{eq:theorem_proof_gj_phi_v1}
\begin{aligned}
    \mathbf{g}_j(\mathbf{x}_{\eta}^*) & \leq  \mathbf{g}_j(\mathbf{x}_{v}^*) + \nabla_\mathbf{x} \mathbf{g}_j(\mathbf{x}_{v}^*)^\top \left(\mathbf{x}_{\eta}-\mathbf{x}_{v}\right)\\
    & \leq \nabla_\mathbf{x} \mathbf{g}_j(\mathbf{x}_{\eta}^*)^\top \left(\mathbf{x}_{\eta}-\mathbf{x}_{v}\right).
    \end{aligned}
\end{equation}
The last inequality follows because $\mathbf{x}_v^*$ is a solution to the saddle-point problem \eqref{eq:saddle_point_problem_regularized_v}, such that $\mathbf{g}_j(\mathbf{x}_v^*) \leq 0$ for all $j$. We then multiply both sides of \eqref{eq:theorem_proof_gj_phi_v1} by $\bm{\lambda}_{\eta,j}^* \geq 0$ and sum up over all $j$, which leads to:
\begin{equation}\label{eq:lambda_g_inequality_1}
\begin{aligned}
    & \bm{\lambda}_{\eta}^\top \mathbf{g}(\mathbf{x}_{\eta}^*) \\
    & \leq \sum_{j=1}^{2N}\nabla_\mathbf{x} \bm{\lambda}_{\eta,j}\cdot \mathbf{g}_j(\mathbf{x}_{\eta}^*)^\top \left(\mathbf{x}_{\eta}^*-\mathbf{x}_{v}^*\right)\\
    & = \nabla_{\mathbf{x}}\mathcal{L}_{\eta}(\mathbf{x}_{\eta}^*,\bm{\lambda}_{\eta}^*)^\top \left(\mathbf{x}_{\eta}^* - \mathbf{x}_{v}^*\right) - \nabla_{\mathbf{x}}\tilde{f}(\mathbf{x}_{\eta}^*)^\top \left(\mathbf{x}_{\eta}^* - \mathbf{x}_{v}^*\right)\\
    & \leq - \nabla_{\mathbf{x}}\tilde{f}(\mathbf{x}_{\eta}^*)^\top \left(\mathbf{x}_{\eta}^* - \mathbf{x}_{v}^*\right),
    \end{aligned}
\end{equation}
where the second inequality is based on the first-order optimality condition, i.e., $\nabla_{\mathbf{x}}\mathcal{L}_{\eta}(\mathbf{x}_{\eta}^*,\bm{\lambda}_{\eta}^*)^\top \left(\mathbf{x}_{\eta}^* - \mathbf{x}_{v}^*\right) \leq 0$.

On the other hand, we have:
\begin{equation}\label{eq:theorem_proof_gj_phi_v2}
    \mathbf{g}_j(\mathbf{x}_{\eta}^*)\geq \mathbf{g}_j(\mathbf{x}_{v}^*) + \nabla_{\mathbf{x}}\mathbf{g}_j(\mathbf{x}_{v}^*)^\top  \left(\mathbf{x}_{\eta}^* - \mathbf{x}_v^*\right).
\end{equation}
Multiplying both sides of \eqref{eq:theorem_proof_gj_phi_v2} by $-\bm{\lambda}_{v,j}^* \leq 0$ and summing up over all $j$ leads to:
\begin{equation}\label{eq:lambda_g_inequality_2}
\begin{aligned}
& -\bm{\lambda}_v^{*\top} \mathbf{g}(\mathbf{x}_{\eta}^*)\\
& \leq -\sum_{j}\bm{\lambda}_{v,j}^* \mathbf{g}_j(\mathbf{x}_v^*) - \sum_{j}\bm{\lambda}_{v,j}^*\cdot \nabla_{\mathbf{x}}\mathbf{g}_j(\mathbf{x}_{v}^*)^\top \left(\mathbf{x}_{\eta}^*-\mathbf{x}_v^*\right)\\
& = \sum_{j}\bm{\lambda}_{v,j}^*\cdot \nabla_{\mathbf{x}}\mathbf{g}_j(\mathbf{x}_{v}^*)^\top \left(\mathbf{x}_v^* - \mathbf{x}_{\eta}^*\right)\\
& = \nabla_{\mathbf{x}}\mathcal{L}_{v}(\mathbf{x}_{v}^*,\bm{\lambda}_{v}^*)^\top \left(\mathbf{x}_{v}^* - \mathbf{x}_{\eta}^*\right) - \nabla_{\mathbf{x}}\tilde{f}(\mathbf{x}_{v}^*)^\top \left(\mathbf{x}_{v}^* - \mathbf{x}_{\eta}^*\right)\\
& \leq \nabla_{\mathbf{x}}\tilde{f}(\mathbf{x}_{v}^*)^\top \left(\mathbf{x}_{\eta}^*-\mathbf{x}_{v}^* \right).
\end{aligned}
\end{equation}
The first equality follows from the condition $(\bm{\lambda}_{v}^*)^\top \mathbf{g}(\mathbf{x}_v^*) = 0$, which holds due to the complementary slackness condition of  \eqref{eq:saddle_point_problem_regularized_v}. The second inequality is obtained from the first-order optimality condition, i.e., $\nabla_{\mathbf{x}}\mathcal{L}_{v}(\mathbf{x}_{v}^*,\bm{\lambda}_{v}^*)^\top \left(\mathbf{x}_{v}^* - \mathbf{x}_{\eta}^*\right) \leq 0$.
% where $(\mathbf{x}_v^*, \bm{\lambda}_{v}^*)$ is the saddle-point of problem \eqref{eq:saddle_point_problem_regularized_v}.

Since the objective function $\tilde{f}(\cdot)$ is strongly convex with a positive constant $c$, 
\begin{equation}\label{eq:tilde_f_strong_convex_nabla}
    \left(\nabla_x\tilde{f}(\mathbf{x}) - \nabla_x\tilde{f}(\mathbf{x}')\right)^\top \left(\mathbf{x} - \mathbf{x}'\right) \leq c\|\mathbf{x} - \mathbf{x}'\|^2_2.
\end{equation}
Substituting \eqref{eq:lambda_g_inequality_1} and \eqref{eq:lambda_g_inequality_2} into \eqref{eq:proof_two_g} and employing \eqref{eq:tilde_f_strong_convex_nabla} yield:
\begin{equation}\label{eq:bound_phiv_v}
    \|\mathbf{x}^*_v - \mathbf{x}_{\eta}^*\|^2_2 \leq \frac{\phi}{2c}\left(\|\bm{\lambda}_v^*\|^2_2 - \|\bm{\lambda}_{\eta}^*\|^2_2\right).
\end{equation}
Finally, substituting \eqref{eq:bound_v} and \eqref{eq:bound_phiv_v} into \eqref{eq:bound_relax} leads to:
\begin{equation}\nonumber
\begin{aligned}
&\|\mathbf{x}^* - \mathbf{x}_{\eta}^*\|^2_2 \\
& ~~~~~~~~\leq \left(\frac{2\left(G_f + G_g \|\bm{\lambda}_{v}^*\|_1\right)}{c}\right)^2 + \frac{\phi}{2c}\left(\|\bm{\lambda}^*_v\|^2_2 - \|\bm{\lambda}_{\eta}^*\|^2_2\right),
\end{aligned}
\end{equation}
which concludes the proof.
\end{proof}

\subsection{Proof of Theorem \ref{thm:online_convergence}}\label{appendix_online_convergence}
\begin{proof}
We begin by investigating the distance between the sequence $\{\mathbf{e}_{\eta,t}\} := \{\mathbf{u}_{\eta,t},\bm{\tau}_{\eta,t},\bm{\lambda}_{\eta,t},\mathbf{z}_{\eta,t}\}$ generated by  \eqref{eq:stochastic_opf_se_controller} at time $t$ and the unique optimizer $\{\mathbf{e}_{\eta,t-1}^*\} := \{\mathbf{u}_{\eta,t-1}^*,\bm{\tau}_{\eta,t-1}^*,\bm{\lambda}_{\eta,t-1}^*,\mathbf{z}_{\eta,t-1}^*\}$ of the saddle-point problem \eqref{eq:saddle_point_problem_regularized_phi_v} at time $t-1$. Based on the definition of the time-varying gradient operator and the non-expansivity property of the project operator, we have:
\begin{equation}\nonumber
\begin{aligned}
     &\|\mathbf{e}_{\eta,t} - \mathbf{e}_{\eta,t-1}^*\|_2 \\ 
     & \leq \|\mathbf{e}_{\eta,t-1} - \epsilon \bm{\Pi}_{\eta,t}(\mathbf{e}_{\eta,t-1}) - \mathbf{e}_{\eta,t-1}^* + \epsilon \bm{\Pi}_{\eta,t}(\mathbf{e}_{\eta,t-1}^*)\|_2. 
     \end{aligned}
\end{equation}
Using the triangle inequality, we obtain:
\begin{equation}\label{eq:proof_4_1}
\begin{aligned}
     &\|\mathbf{e}_{\eta,t} - \mathbf{e}_{\eta,t-1}^*\|_2 \\ 
     & \leq \|\mathbf{e}_{\eta,t-1} - \mathbf{e}_{\eta,t-1}^*\|_2 +  \epsilon^2\|\bm{\Pi}_{\eta,t}(\mathbf{e}_{\eta,t-1}) -\bm{\Pi}_{\eta,t}(\mathbf{e}_{\eta,t-1}^*)\|_2 \\
     & ~~~~- 2\epsilon \left(\bm{\Pi}_{\eta,t}(\mathbf{e}_{\eta,t-1}) -\bm{\Pi}_{\eta,t}(\mathbf{e}_{\eta,t-1}^*)\right)^\top \left(\mathbf{e}_{\eta,t-1} - \mathbf{e}_{\eta,t-1}^*\right)\\
     & \leq \alpha\|\mathbf{e}_{\eta,t-1} - \mathbf{e}_{\eta,t-1}^*\|_2,
     \end{aligned}
\end{equation}
where $\alpha = \sqrt{1 - 2\epsilon\tilde{M} + \epsilon^2\tilde{L}^2}$. The last inequality is due to the strongly monotone and Lipschitz properties of $\bm{\Pi}_{\eta,t}$, shown in \eqref{eq:property_Monotone_stochastic} and \eqref{eq:property_Lipschitz_stochastic}. Now we are ready to show the convergence of the online gradient updates in \eqref{eq:stochastic_opf_se_controller}. For $t>0$, the distance between the sequence $\mathbf{e}_{\eta,t}$ generated by the gradient updates in \eqref{eq:stochastic_opf_se_controller} and the unique saddle point $\mathbf{e}_{\eta,t}^*$ of the optimization problem \eqref{eq:saddle_point_problem_regularized_phi_v} is bounded by
\begin{equation}\nonumber
\begin{aligned}
 \|\mathbf{e}_{\eta,t} - \mathbf{e}_{\eta,t}^*\|_2 & = \|\mathbf{e}_{\eta,t} - \mathbf{e}_{\eta,t}^* + \mathbf{e}_{\eta,t-1}^* - \mathbf{e}_{\eta,t-1}^*\|_2\\
    & \leq \|\mathbf{e}_{\eta,t} - \mathbf{e}_{\eta,t-1}^*\|_2 + \| \mathbf{e}_{\eta,t}^* - \mathbf{e}_{\eta,t-1}^* \|_2\\
    & \leq \alpha\|\mathbf{e}_{\eta,t-1} - \mathbf{e}_{\eta,t-1}^*\|_2 + \sigma_\mathbf{e}.
\end{aligned}
\end{equation}
The last inequality follows from \eqref{eq:proof_4_1} and the difference between the optimization solutions of consecutive time instants, as shown in Assumptions \ref{ass:bound_primal} and \ref{ass:bound_constraints}. We then recursively implement the above inequality until $t = 0$ resulting in:
\begin{equation}
\begin{aligned}\label{eq:proof_online_difference_t}
 \|\mathbf{e}_{\eta,t} - \mathbf{e}_{\eta,t}^*\|_2
    & \leq \alpha^t\|\mathbf{e}_{\eta,0} - \mathbf{e}_{\eta,0}^*\|_2 + \left(\frac{1-\alpha^t}{\alpha}\right)\sigma_\mathbf{e}.
\end{aligned}
\end{equation}
Choosing the step size as $0 < \epsilon <\frac{2\widetilde{M}}{\widetilde{L}^2}$ from Lemma \ref{lma:step_size_stochastic} leads to $0 < \alpha < 1$. As $t \to \infty$, the term $\alpha^t$ on the right-hand-side of \eqref{eq:proof_online_difference_t} will vanish. Given such $\alpha$ and any initial point $\mathbf{e}_{\eta,0}$ located in the feasible set $\mathcal{X}_0$, we let the gradient update \eqref{eq:stochastic_opf_se_controller} run over time as $t \to \infty$, the difference is bounded by
\begin{equation}\label{eq:online_bound}
\begin{aligned}\nonumber
\lim_{t\to\infty} \sup \|\mathbf{e}_{\eta,t} - \mathbf{e}_{\eta,t}^*\|_2 = \frac{\sigma_\mathbf{e}}{\sqrt{1 - 2\epsilon\widetilde{M} + \epsilon^2\widetilde{L}^2}},
\end{aligned}
\end{equation}
which concludes the proof.
\end{proof}

\end{document}